\documentclass{amsart}

\usepackage{amsmath, amssymb, amsthm, amsfonts, color}

\input xy
\xyoption{all}

\usepackage{hyperref}

\swapnumbers
\numberwithin{equation}{section}

\theoremstyle{plain}
\newtheorem{theorem}[subsubsection]{Theorem}
\newtheorem{lemma}[subsubsection]{Lemma}

\newtheorem{cor}[subsubsection]{Corollary}
\newtheorem{conj}[subsubsection]{Conjecture}

\theoremstyle{definition}
\newtheorem{defn}[subsubsection]{Definition}

\newtheorem{remark}[subsubsection]{Remark}
\newtheorem{exam}[subsubsection]{Example}

\def\AA{\mathbb{A}}

\def\FF{\mathbb{F}}
\def\GG{\mathbb{G}}

\def\QQ{\mathbb{Q}}
\def\RR{\mathbb{R}}
\def\SS{\mathbb{S}}
\def\TT{\mathbb{T}}

\def\WW{\mathbb{W}}

\def\ZZ{\mathbb{Z}}

\newcommand\cF{\mathcal{F}}
\newcommand\cG{\mathcal{G}}

\newcommand\cK{\mathcal{K}}
\newcommand\cL{\mathcal{L}}

\newcommand\cN{\mathcal{N}}
\newcommand\cO{\mathcal{O}}

\newcommand\cS{\mathcal{S}}

\def\bI{\mathbf{I}}
\def\bJ{\mathbf{J}}
\def\bK{\mathbf{K}}

\def\bP{\mathbf{P}}
\def\bQ{\mathbf{Q}}

\def\bT{\mathbf{T}}

\newcommand\frA{\mathfrak{A}}
\newcommand\frB{\mathfrak{B}}

\newcommand\frD{\mathfrak{D}}

\newcommand\frI{\mathfrak{I}}
\newcommand\frJ{\mathfrak{J}}

\newcommand\frP{\mathfrak{P}}

\newcommand\frc{\mathfrak{c}}

\newcommand\frg{\mathfrak{g}}

\newcommand\frl{\mathfrak{l}}
\newcommand\fm{\mathfrak{m}}
\newcommand\frn{\mathfrak{n}}
\newcommand\frp{\mathfrak{p}}
\newcommand\frq{\mathfrak{q}}

\newcommand\frs{\mathfrak{s}}
\newcommand\frt{\mathfrak{t}}

\newcommand\frz{\mathfrak{z}}

\newcommand\AS{\textup{AS}}
\newcommand\Av{\textup{Av}}

\newcommand{\ch}{\textup{char}}

\newcommand\ev{\textup{ev}}

\newcommand{\fl}{f\ell}

\newcommand\FT{\mathrm{FT}}

\newcommand\Gal{\textup{Gal}}

\newcommand{\ind}{\textup{ind}}

\newcommand\Lie{\textup{Lie}\ }

\newcommand\nil{\textup{nil}}
\newcommand{\Nm}{\textup{Nm}}

\newcommand\Out{\textup{Out}}

\newcommand\pr{\textup{pr}}

\newcommand\res{\textup{res}}

\newcommand\Spec{\textup{Spec}\ }

\newcommand\Stab{\textup{Stab}}

\newcommand{\Tr}{\textup{Tr}}
\newcommand{\tr}{\textup{tr}}

\newcommand{\val}{\textup{val}}

\DeclareMathOperator{\colim}{colim}
\DeclareMathOperator{\CS}{CS}

\newcommand\Hom{\textup{Hom}}

\newcommand\Mor{\textup{Mor}}

\newcommand\uHom{\underline{\Hom}}

\newcommand\SL{\textup{SL}}
\renewcommand\sl{\mathfrak{sl}}

\newcommand{\Gm}{\GG_m}

\newcommand{\Ad}{\textup{Ad}}

\newcommand{\der}{\textup{der}}

\newcommand\xch{\mathbb{X}^*}
\newcommand\xcoch{\mathbb{X}_*}

\newcommand\upH{\textup{H}}

\newcommand{\Qlbar}{\overline{\QQ}_\ell}
\newcommand{\const}[1]{\overline{\QQ}_{\ell,#1}}

\renewcommand{\j}[1]{\langle{#1}\rangle}
\newcommand{\wt}[1]{\widetilde{#1}}
\newcommand{\wh}[1]{\widehat{#1}}
\newcommand\quash[1]{}
\newcommand\un{\underline}
\newcommand{\bu}{\bullet}
\newcommand{\ov}{\overline}

\newcommand{\tl}[1]{[\![#1]\!]}
\newcommand{\lr}[1]{(\!(#1)\!)}
\newcommand\sss{\subsubsection}

\newcommand\op{\oplus}
\newcommand\ot{\otimes}

\newcommand{\sslash}{\mathbin{/\mkern-6mu/}}
\renewcommand\c\circ
\newcommand\vn{\varnothing}

\newcommand\mt{\mapsto}
\newcommand{\incl}{\hookrightarrow}
\newcommand{\isom}{\stackrel{\sim}{\to}}
\newcommand{\bij}{\leftrightarrow}
\newcommand{\surj}{\twoheadrightarrow}

\newcommand{\orr}{\overrightarrow}

\newcommand\xr{\xrightarrow}
\newcommand{\ari}{\ar@{^{(}->}} 

\renewcommand\a\alpha
\renewcommand\b\beta
\newcommand\g\gamma
\newcommand\G\Gamma
\renewcommand\d\delta
\newcommand\D\Delta
\newcommand{\e}{\epsilon}

\renewcommand\r\rho
\newcommand{\io}{\iota}

\renewcommand{\t}{t_E}

\newcommand{\y}{\eta}
\newcommand{\z}{\zeta}

\renewcommand{\l}{\lambda}
\renewcommand{\L}{\Lambda}
\newcommand{\om}{\omega}

\newcommand\hs{\heartsuit}

\renewcommand\v{\vee}

\newcommand\hZ{\widehat{\ZZ}}

\newcommand\PD{\textup{PD}}
\newcommand\PCD{\textup{PCD}}
\newcommand\CD{\textup{CD}}
\newcommand\ECD{\textup{ECD}}
\newcommand\OS{\textup{OS}}
\newcommand\cInd{\textup{cInd}}
\newcommand\Yu{\textup{Yu}}

\title{Character sheaves on loop Lie algebras: polar partition}

\dedicatory{To G\'erard Laumon with deep gratitude}
\author{Bao Ch\^au Ng\^o}
\address{Department of Mathematics, University of Chicago, Chicago, IL}
\email{ngo@uchicago.edu}
\author{Zhiwei Yun}
\thanks{N.B.C. is supported partially by the NSF. Z.Y. is supported partially by the Simons Foundation.}
\address{Department of Mathematics, Massachusetts Institute of Technology, Cambridge, MA}
\email{zyun@mit.edu}
\date{}

\begin{document}

\begin{abstract}
In this paper we propose a partition of loop Lie algebras into invariant sets parametrized by polar data. Our polar partition is motivated by J-K. Yu's construction of supercuspidal representations and leads to a conjectural construction of character sheaves on loop Lie algebras. 
\end{abstract}

\maketitle

\tableofcontents

\section{Introduction}
For a connected reductive group $G$ over a field $k$, Lusztig introduced a class of $G$-equivariant simple perverse sheaves on $G$ called {\em character sheaves}. When $k=\FF_{q}$ is a finite field, the pointwise Frobenius traces of these sheaves are closely related to irreducible characters of the finite Chevalley group $G(\FF_{q})$. 

On the other hand, Lusztig \cite{L-Fourier} and Mirkovi\'c \cite{mirkovic_character_2004} introduce analogs of character sheaves on the Lie algebra $\frg$. It turns out that character sheaves on $\frg$ are exactly Fourier transforms of $G$-equivariant simple perverse sheaves supported on a single co-adjoint orbit closure, which can be used as a definition of character sheaves on $\frg$. 

Passing from a finite field to a local function field $F=\FF_q\lr{t}$, a theory of character sheaves on the loop group $LG=G\lr{t}$ has been long expected. They are supposed to be related to distributional characters of $G(F)$ in a way similar to finite Chevalley groups. The construction of character sheaves on the loop group $LG$ is an active research topic, see  \cite{L-unip}, \cite{BKV0} and \cite{BC}, to name a few, although the picture is far from complete.

In this article, we propose a definition of character sheaves on the loop Lie algebra $L\frg=\frg\lr{t}$. We replace $\FF_q$ by an algebraically closed field $k$ of characteristic $p$, and still use $F$ to denote $k\lr{t}$. The proposed character sheaves form a full subcategory $\CS(L\frg)$ of loop group equivariant sheaves on $L\frg$. The rough idea is to mimic the Fourier transform definition of character sheaves on $\frg$ mentioned above. However, the co-adjoint orbits in the dual loop algebra $L\frg^*$ are too "thin" to yield reasonable notion of character sheaves after Fourier transform; we need to thicken the co-adjoint orbits. Making the thickening precise led to the notion of {\em polar data} that is crucial for our proposal.   A polar datum $(M,\l)$ consists of a twisted Levi subgroup $M\subset G$ over $F$,  and $\l$ is a certain Laurent tail in the dual of the center of $\fm=\Lie M$. In Section 2, we define the notion of polar data and construct the partition of the loop space of the dual vector space $\frg^*$ of the Lie algebra $\frg$ by polar data $(M,\lambda)$ (Theorem \ref{th:partition c*}). We explain how our polar data are related to the data that J-K. Yu used in his construction of supercuspidal representations, see \cite{Yu} and \cite{fintzen_construction_2021}.

In Section 3, for each polar datum $(M,\l)$, or rather a more refined notion called polar-cuspidal datum $\xi=(M,\l, L,\cO,\cL)$, we explain how to define a full subcategory $\CS_{\xi}(L\frg)$ inside the dg category of $LG$-equivariant sheaves on $L\frg$ via Fourier transform. The category $\CS(L\frg)$ turns out to be the direct sum of $\CS_{\xi}(L\frg)$ for various polar-cuspidal data $\xi$ (Theorem \ref{th:block}), which can be viewed as a linearized geometric analog of the Bernstein decomposition. Moreover, in Section \ref{ss:compare Yu}, we exhibit in a precise way the parallelism between our character sheaves and characters of J.K.Yu's supercuspidal representations, giving justification for the nomenclature.


Because the foundational materials necessary to establish properties of the Fourier transform of equivariant perverse sheaves on loop Lie algebras are still being worked out, full proofs of theorems stated in Section 3 will appear elsewhere. Further study of geometric properties of these character sheaves will rely on techniques developed by Bouthier--Kazhdan--Varshavsky \cite{BKV} in their study of the affine Springer sheaf.


\subsection*{Acknowledgement} 

Our mathematical works have been profoundly influenced by Laumon's ideas. This paper is no exception as it aims at being an application of Laumon's theorem on the Fourier-Deligne transform to representation theory. We dedicate this work to him with our deepest gratitude.

The authors thank A.Bouthier for helpful comments.

\section{Polar data}

After setting up notations for discussing loop groups and loop Lie algebras, we define the key new notion of polar data and polar subsets of a loop Lie algebra.

\subsection{Local function fields}
Let $k$ be an algebraically closed field of characteristic $p$, and $F=k\lr{t}$ be the field of Laurent series in one variable $t$ over $k$. Write $\cO_F=k\tl{t}$ and $\fm_F=t\cO_F$. Let $\om(F)=F dt$ be the one-dimensional $F$-vector space of K\"ahler differentials of $F$ over $k$, and let $\om(\cO_F)=\cO_F dt$. We have a residue map
\begin{equation*}
\res_F: \om(F)\to k
\end{equation*} 
satisfying $\res_{F}(\frac{dt}{t})=1$.

For any tamely ramified finite extension $E/F$ with uniformizer $t_E$ (for example we may take $t_E=t^{1/m}$ for $m=[E:F]$) , we have $\om(E)=Ed t_E=Edt$, and $\om(\cO_{E})=\cO_{E}dt_E$. We normalize the residue map $\res_{E}: \om(E)\to k$ so that $\res_{E}(\frac{dt}{t})=1$, equivalently, $\res_{E}(\frac{dt_E}{t_E})=[E:F]^{-1}$. 

Let $F_{\infty}=\cup_{(p,n)=1}k\lr{t^{1/n}}$ be a maximally  tamely ramified extension of $F$. Then the Galois group 
\begin{equation*}
\Gal(F_{\infty}/F)\cong \prod_{(p,n)=1}\mu_{n}(k)=:\hZ'(1).
\end{equation*}
We denote the valuation on $F_{\infty}$ by
\begin{equation*}
\val: F^{\times}_{\infty}\to \QQ
\end{equation*}
normalized by $\val(t)=1$. Let
\begin{equation*}
\om(F_{\infty}):=\bigcup_{E\subset F_{\infty}, [E:F]<\infty}\om(E), \quad \om(\cO_{F_{\infty}})=\bigcup_{E\subset F_{\infty}, [E:F]<\infty}\om(\cO_{E}).
\end{equation*}
The residue maps for finite extensions $E\subset F_\infty$ are compatible with inclusions, and they together define a residue map on $\om(F_{\infty})$
\begin{equation*}
\res: \om(F_{\infty})\to k.
\end{equation*}

\subsection{Tamely ramified reductive groups}
Let $G$ be a connected reductive group over $F$ that splits over a tamely ramified extension of $F$. 
Let $G_0$ be a split reductive group over $k$ such that there exists an isomorphism $G\times_F F^s \simeq G_0\times_k F^s$, which gives rise to an element ${\rm H}^1(F,{\rm Aut}(G_0(F^s)))$. Recall the exact sequence of automorphisms of $G_0(F^s)$ 
$$1\to G_0^{\rm ad}(F^s) \to {\rm Aut}(G_0(F^s)) \to {\rm Out}(G_0(F^s)) \to 1$$
where ${\rm Out}(G_0)(F^s)={\rm Out}(G_0)$ is the discrete group of outer automorphisms of $G_0$. Te map ${\rm Aut}(G_0(F^s)) \to {\rm Out}(G_0)$ induces a map
$${\rm H}^1(F,{\rm Aut}(G_0(F^s))) \to {\rm H}^1(F,{\rm Out}(G_0(F^s)))=\Hom(\hZ'(1), {\rm Out}(G_0)).$$
We claim this map is a bijection. Indeed, since $F$ is a $C_{1}$-field, the above map is injective; it is also surjective because one has a section  ${\rm Out}(G_0)\to {\rm Aut}(G_0)$ given by pinned automorphisms of $G_0$ with respect to some pinning. 

In particular, $G$ is quasi-split. We note that even if one is only interested in split groups, we will encounter in the polar partition twisted Levi subgroups that are generally nonsplit.

Let $\frg=\Lie G$.




\sss{Abstract Weyl group scheme} Let $(\WW,\SS)$ be the {\em abstract Weyl group scheme} of $G$ defined as follows. It is a finite \'etale group scheme over $F$. Its $F^{s}$-points $\WW(F^{s})$ is the set of $G_{F^{s}}$-orbits on $\fl_{G, F^{s}}\times \fl_{G, F^{s}}$ where $\fl_G$ is the flag variety of $G$. The action of $\Gal(F^{s}/F)$ on $\WW(F^{s})$ factors through the tame quotient $\hZ'(1)$.  

We have a length function 
\begin{equation*}
\ell: \WW\to \ZZ_{\ge0}
\end{equation*}
defined on $\WW(F^{s})$ by $\ell(w)=\dim O(w)-\dim G$, where $O(w)\subset \fl_{G}\times \fl_{G}$ is the $G_{F^{s}}$-orbit indexed by $w\in \WW(F^{s})$.
Let $\SS=\ell^{-1}(1)\subset \WW$ be the subscheme of length $1$. Then $(\WW(F^{s}),\SS(F^{s}))$ has a canonical structure of a Coxeter group. The action of $\hZ'(1)$ on $\WW(F^{s})$ permutes $\SS(F^{s})$. Altogether we get a {\em Coxeter group scheme} $(\WW,\SS)$ over $F$ attached to $G$.

We make the following assumption throughout the paper:
\begin{equation}\label{char ass}
    \mbox{The characteristic $p$ of $k$ is coprime to $|\WW|$.}
\end{equation}



\sss{Abstract Cartan} The reductive quotients of all Borel $F$-subgroups of $G$ are canonically identified. This gives an $F$-torus $\TT$ that is the {\em abstract Cartan of $G$}. There is a canonical action of $\WW$ on $\TT$. For $w\in \WW$, and $(B_{1}, B_{2})\in O(w)$, the image of $B_{1}\cap B_{2}\to \TT\times \TT$ (projections to the reductive quotients of $B_{1}$ and $B_{2}$) is the graph of the $w$-action on $\TT$.

Let $(T\subset B)$ be a maximal torus contained in a Borel subgroup of $G$, both defined over $F$. The choice of $B$ induces an isomorphism $\io_{T,B}: T\isom \TT$, which also induces an isomorphism between the Weyl group schemes $W(G,T)\isom \WW$, compatible with their actions on $T$ and $\WW$. 

\sss{Abstract based root datum} 
For a pair $(T\subset B)$ defined over $F$,  the roots $\Phi(G_{F_\infty},T_{F_\infty})$ with the action of $\Gal(F_\infty/F)$ can be viewed as a finite \'etale $F$-scheme $\un\Phi(G,T)$, with a subscheme $\un\D=\un\D(G,B,T)$ of simple roots. Via the isomorphism $\io_{T,B}$ they give rise to a based root scheme $\un\D\subset \un\Phi\subset\uHom_{F}(\TT, \Gm)$ of the abstract Cartan $\TT$. The pair $(\un\Phi,\un\D)$ is independent of the choice of $(T\subset B)$.

Sending $\a\in \D$ to its reflection $s_{\a}\in \WW$ gives an isomorphism of $F$-schemes $\un\D\cong \SS$.

There is a canonical map $\un\Phi\to \uHom_{F}(\Gm, \TT)$ sending $\a\mt \a^{\v}$ for $\a\in \un\Phi(F^{s})$. Denote the image of this map by $\un\Phi^{\v}$. This is the abstract coroot scheme. Taking the image of $\un\D$ we get the abstract simple coroot scheme $\un\D^{\v}$.



\sss{The Chevalley base}\label{ss:c}

Let $\frc=\frg\sslash G$, an affine scheme over $F$. Let $\chi: \frg\to \frc$ be the natural map. It is easy to see that  $\frc$ carries a canonical integral model $\un\frc$ over $\cO_{F}$. Indeed, let $E/F$ be a finite tame extension that splits $G$, so that $G_E\cong G_{0,E}$ where $G_0$ is a connected reductive group over $k$. The descent datum for $G$ gives a homomorphism
\begin{equation*}
\r_{G,\Out}: \Gal(E/F)\to \Out(G_0).
\end{equation*} 
Let $\frg_0=\Lie G_0$, and denote by $\frc_{0}=\frg_0\sslash G_0$ over $k$. Since $\Out(G_0)$ acts on $\frc_0$, we get an action of $\Gal(E/F)$ on $\frc_0$ via $\r_{G,\Out}$. Then define
\begin{equation*}
\un\frc=\left(R_{\cO_{E}/\cO_{F}}\frc_{0,\cO_{E}}\right)^{\Gal(E/F)}.
\end{equation*}


Consider the fixed points
\begin{equation*}
\frc_{0}^{\flat}:=\frc_{0}^{\Gal(E/F)}
\end{equation*}
From the construction we have a canonical surjection from the special fiber of $\un\frc$
\begin{equation}\label{c red quot}
\un\frc\times_{\Spec \cO_{F}}\Spec k\to \frc_{0}^{\flat}
\end{equation}
which may not be an isomorphism.  Indeed, the special fiber of $\un\frc$ has dimension equal to the rank of $G$, while $\frc_{0}^{\flat}$ has dimension equal to the split rank of $G$.

\sss{Dual Chevalley base}

Let 
\begin{equation*}
    \frg^*=\Hom_F(\frg, \om(F))
\end{equation*}
be the dual Lie algebra of $F$. Using the $F$-basis $dt$ of $\om(F)$ we can identify $\frg^{*}$ with the $F$-linear dual to $\frt$. The residue pairing gives a non-degenerate $k$-linear pairing
\begin{equation}\label{res pairing g}
\res\j{\cdot,\cdot}: \frg\times \frg^{*}\to k.
\end{equation}

Let $\frc^{*}=\frg^{*}\sslash G$. Since $\Out(G_0)$ also acts on $\frc^{*}_0=\frg_0^*\sslash G_0$, the same constructions of $\un\frc$ and $\frc_{0}^{\flat}$ gives a canonical model $\un\frc^{*}$ of $\frc^{*}$ over $\cO_{F}$, an affine $k$-scheme $\frc_{0}^{*,\flat}$ and a surjective map
\begin{equation}\label{c* red quot}
\un\frc^{*}\times_{\Spec \cO_{F}}\Spec k\to\frc_{0}^{*,\flat}. 
\end{equation}

\sss{Tori}\label{ss:tori}

Let $T\subset G$ be a torus over $F$. The $F$-torus $T$ has a unique parahoric group scheme, whose arc space gives the unique parahoric subgroup $\bT\subset LT$. Let $\bT^+$ be the pro-unipotent radical of $\bT$. Their Lie algebras define two $\cO_F$-lattices in $\frt$
\begin{equation*}
\frt_c:=\Lie \bT, \quad \frt_{tn}=\Lie \bT^+.
\end{equation*}


%
%

In the dual Lie algebra
\begin{equation*}
\frt^{*}:=\Hom_{F}(\frt, \om(F)),
\end{equation*}
using the residue pairing \eqref{res pairing g} for $\frg=\frt$, we define the annihilators of $\frt_c$ and $\frt_{tn}$ to be
\begin{equation*}
\frt^*_c:= (\frt_{tn})^{\bot}, \quad \frt^*_{tn}=\frt_c^{\bot}.
\end{equation*}
These are both $\cO_F$-lattices in $\frt^*$.
Via the residue pairing, the torsion $\cO_{F}$-module $\frt^{*}/\frt^*_{tn}$ is identified with the continuous $k$-linear dual of $\frt_c$, under the $t$-adic topology.

\begin{exam}\label{ex:split torus} Suppose $T$ is a split torus, so that $T=T_{0}\times_{\Spec k}\Spec F$ for a torus $T_{0}$ over $k$ with Lie algebra $\frt_{0}$ over $k$. Then under the identifications $\frt=\frt_{0}\ot_{k}F$ and $\frt^{*}=\frt^{*}_{0}\ot_{k}\om_{F}$, we have
\begin{eqnarray*}
\frt_c=\frt_{0}\ot_{k}\cO_{F},\\
\frt_{tn}=\frt_{0}\ot_{k}\fm_{F},\\
\frt^*_c=\frt^{*}_{0}\ot_{k}\cO_{F}\frac{dt}{t},\\
\frt^*_{tn}=\frt^{*}_{0}\ot_{k}\cO_{F}dt,\\
\label{split torus F/O}\frt^{*}/\frt^*_{tn}=\frt^{*}_{0}\ot_{k}(\om(F)/\om(\cO_{F})).
\end{eqnarray*}
\end{exam}

The above discussions apply to any finite extension $E/F$ and to the base change $T_E$. We write $\frt(E)$ for the Lie algebra of $T_E$, inside which we have $\cO_E$-lattices
\begin{equation*}
    \frt(E)_{tn}\subset \frt(E)_c.
\end{equation*}
Dually, let $\frt^*(E)=\Hom_E(\frt(E), \om(E))$, inside which we have $\cO_E$-lattices
\begin{equation*}
    \frt^*(E)_{tn}\subset \frt^*(E)_c.
\end{equation*}





\sss{Maximal tori in $G$}\label{sss: max tori}
The maximal tori in $G$ up to conjugacy are classified by the Galois cohomology group $\upH^{1}(\Gal(F^s/F),\WW(F^s))$. With the assumption \eqref{char ass} on the characteristic of $k$, all maximal tori of $G$ split over a tame extension of $F$, and their conjugacy classes are classified by  $\upH^{1}(\Gal(F_\infty/F),\WW(F_\infty))$, with the base point corresponding to any maximal torus contained in a Borel subgroup of $G$ (such a maximal torus is isomorphic to the abstract Cartan $\TT$).

Now suppose a maximal torus $T\subset G$ corresponds to the class of a cocycle  $w: \hZ'(1)\to \WW(F_{\infty})$. Choose a finite tame extension $E/F$ such that $T_{E}$ is split, then $w$ factors through $\Gal(E/F)\cong \mu_{m}$ where $m=[E:F]$. Then
\begin{equation*}
    T(F)=\{x\in \TT(E)|\z(x)=w(\z)x, \forall \z\in \mu_m\}.
\end{equation*}
Here $\z(x)$ refers to the action of $\Gal(F_{\infty}/F)$ on $\TT(F_{\infty})$ coming from its $F$-form $\TT$.  

\begin{exam}\label{ex:max tori in split G}
    Suppose $G$ is split over $F$, then $\TT$ is split. Write $\frt_0=\xcoch(\TT)\ot k$, with the natural action of $\WW$. Let $T\subset G$ be a maximal torus corresponds to the class of a cocycle, now a homomorphism,  $w: \hZ'(1)\to \WW=\WW(F_{\infty})$. It uniquely factors through an injection $\ov w: \mu_m\incl \WW$ for some $m$, via which $\mu_m$ acts on $\frt_0$. Let $E=k\lr{t_E}$ with $t_E=t^{1/m}$. Then 
\begin{equation*}
\frt^{*}=\frt^{*}(E)^{\mu_{m}}=(\frt^{*}_{0}\ot_{k} \om(E))^{\mu_{m}}, \quad \frt^*_{tn}=(\frt^{*}_{0}\ot_{k} \om(\cO_{E}))^{\mu_{m}}.
\end{equation*}
Here $\mu_m$ acts diagonally on both $\frt_{0}^{*}$ and $\om(E)=Edt$. 

    Decompose $\frt^*_0$ into weight spaces under the $\mu_m$-action via $w$:
\begin{equation}\label{t0 eigen}
    \frt^*_0=\bigoplus_{i\in \ZZ/m\ZZ}\frt^*_0(i).
\end{equation}
Then $\frt^{*}$ is the $t$-adic completion of
\begin{equation*}
 \bigoplus_{i\in \ZZ}\frt_{0}^{*}(i)\frac{d\t}{\t^{i+1}}.
\end{equation*}
We also have
\begin{equation*}
\frt^*_{tn}=\prod_{i\in \ZZ_{\ge0}}\frt_{0}^{*}(-i-1)\t^{i}d\t.
\end{equation*}
Therefore we have a canonical isomorphism
\begin{equation}\label{colattice}
\frt^{*}/\frt^*_{tn}\cong \bigoplus_{i\in \ZZ_{\ge0}}\frt_{0}^{*}(i)\frac{d\t}{\t^{i+1}}.
\end{equation}
\end{exam}




\sss{Root system as a scheme}\label{sss:rs}
Let $T\subset G$ be a maximal torus. The root system $\Phi(G_{F_{\infty}},T_{F_\infty})$ carries a continuous action of $\Gal(F_\infty/F)$, thus can be viewed as a finite \'etale $F$-scheme that we denote by $\un\Phi(G,T)$. The Weyl group scheme $W(G,T)=N_{G}(T)/T$ (as a finite \'etale group scheme over $F$) acts on $\un\Phi(G,T)$.

Similarly we define a finite \'etale $F$-scheme of coroots the $\un\Phi^\v(G,T)$ with a $W(G,T)$-action.

\subsection{Loop group and loop Lie algebra}\label{sss:loop Lie}

For an affine $F$-scheme $X$, let $LX$ be its loop space as an ind-scheme over $k$, i.e., $LX(R)=X(R\lr{t})$ for any $k$-algebra $R$. For an $\cO_F$-scheme $Y$, let $L^+Y$ be its arc space as a scheme over $k$, i.e., $L^+Y(R)=Y(R\tl{t})$ for any $k$-algebra $R$.

We can therefore form the loop group $LG$, and the loop Lie algebra $L\frg$, the latter being an increasing union of infinite-dimensional affine spaces over $k$.

We still denote by $\chi$ (resp. $\chi^{*}$) the loop functor applied to $\chi:\frg\to \frc$ and $\chi^{*}:\frg^{*}\to \frc^{*}$
\begin{equation*}
\chi: L\frg\to L\frc, \quad \chi^{*}: L\frg^{*}\to L\frc^{*}.
\end{equation*}

\sss{Parahoric subgroups}
Bruhat and Tits defined a class of $\cO_F$-models of $G$ called {\em parahoric group schemes}. A parahoric subgroups $\bP\subset LG$ is of the form $L^+\cG$, where $\cG$ is a parahoric group scheme of $G$ over $\cO_F$. The Lie algebra of a parahoric subgroup $\bP\subset LG$ is a parahoric subalgebra of $L\frg$. A minimal parahoric subgroup of $LG$ is called an Iwahori subgroup, whose Lie algebra is called an Iwahori subalgebra of $L\frg$. 

\sss{Compact and topologically nilpotent loci}
Using the canonical $\cO_{F}$-model $\un\frc$ of $\frc$ as constructed in \S\ref{ss:c}, we define the closed subscheme
\begin{equation*}
(L\frc)_{c}:=L^{+}\un\frc\subset L\frc.
\end{equation*}
The subscript $(-)_c$ stands for compact. Alternatively, $(L\frc)_{c}$ is the image under $\chi$ of any Iwahori subalgebra $\frI\subset L\frg$. 

Let 
\begin{equation*}
(L\frc)_{tn}\subset (L\frc)_{c}
\end{equation*}
be the fiber over $0$ of the evaluation map to the special fiber of $\un\frc$ composed with the natural map \eqref{c red quot}
\begin{equation}\label{Lc red quot}
\e_{\frc}: (L\frc)_{c}=L^{+}\un\frc\to \un\frc\times_{\Spec \cO_{F}}\Spec k\xr{\eqref{c red quot}}\frc_{0}^{\flat}.
\end{equation}
Here the subscript $(-)_{tn}$ stands for {\em topologically nilpotent}. Alternatively, $(L\frc)_{tn}$ is  the image under $\chi$ of $\frI^{+}$, where $\frI^{+}\subset \frI$ is the pro-nilpotent radical of any Iwahori subalgebra $\frI$.

Let $(L\frg)_{tn}\subset (L\frg)_{c}\subset L\frg$ be the preimages of $(L\frc)_{tn}$ and $(L\frc)_{c}$ respectively. The $k$-points of $(L\frg)_{tn}$ (resp. $(L\frg)_{c}$) are exactly the topologically nilpotent (resp. compact) elements of $\frg(F)$. 

Similarly, using the canonical $\cO_{F}$-model $\un\frc^{*}$ of $\frc^{*}$ and the map \eqref{c* red quot}, we similarly define closed subschemes 
\begin{equation*}
(L\frc^{*})_{tn}\subset (L\frc^{*})_{c}\subset L\frc^{*}.
\end{equation*}
Alternatively, $(L\frc^{*})_{tn}$ (resp. $(L\frc^{*})_{c}$) is the image of $(\Lie\bI)^{\bot}$ (resp. $(\Lie\bI^{+})^{\bot}$) for any Iwahori subgoup $\bI\subset LG$. 

Let $(L\frg^{*})_{tn}$ (resp. $(L\frg^{*})_{c}$) be the preimage of $(L\frc^{*})_{tn}$ (resp. $(L\frc^{*})_{c}$) under $\chi^{*}: L\frg^{*}\to L\frc^{*}$. Then $(L\frg^{*})_{tn}$ is the union of $\frI^{\bot}$ for all Iwahori subalgebras $\frI$.

\sss{Remark}
    When $G=T$ is a torus, the subschemes $(L\frt^{*})_{tn}$ and $(L\frt^*)_c$ are the arc spaces of the lattices $\frt^*_{tn}$ and $\frt^*_c$ (as affine spaces over $\cO_{F}$) respectively.



\subsection{Toral polar data}
Let $T\subset G$ be a maximal torus and let $E/F$ be a finite tame  extension that splits $T$. For any coroot $\a^{\v}\in \Phi^{\v}(G_{E}, T_{E})$, its differential $d\a^{\v}$ gives an element in $\frt(E)_c$. Pairing with $d\a^\v$ induces an $\cO_{E}$-linear map 

\begin{equation*}
d\a^{\v}: \frt^{*}(E)/\frt^*(E)_{tn}\to \om(E)/\om(\cO_{E}).
\end{equation*}

\begin{defn} Let $T\subset G$ be a maximal torus. An element $\l\in \frt^{*}/\frt^*_{tn}$ is called {\em $G$-regular} if for some finite tame extension $E/F$ that splits $T$, and 
any $\a^{\v}\in \Phi^{\v}(G_{E}, T_{E})$, we have $d\a^{\v}(\l)\ne 0\in \om(E)/\om(\cO_{E})$.
\end{defn}
It is easy to check that the notion of $G$-regularity is independent of the choice of the splitting field $E$ of $T$.

\begin{defn} A {\em toral polar datum} for $G$ is a pair $(T,\l)$, where $T$ is a maximal $F$-torus of $G$ with Lie algebra $\frt$, and $\l\in \frt^{*}/\frt^*_{tn}$ is $G$-regular. 
\end{defn}
There is an obvious action of $G(F)$ on the set of toral polar data by conjugating $T$ and the co-adjoint action on $\l$.

\begin{exam}[Split toral polar data] Suppose $T=T_{0}\times_{\Spec k}\Spec F$ is split. Under the isomorphism \eqref{split torus F/O}, $\l\in \frt^{*}/\frt^*_{tn}$ has a unique representative as a Laurent tail
\begin{equation*}
\l=(\l_{-n}t^{-n}+\l_{-n+1}t^{-n+1}+\cdots+\l_{0})\frac{dt}{t}
\end{equation*}
where $\l_{i}\in \frt^{*}_{0}$, and $n\in \ZZ_{\ge0}$. If $\l_{-n}\ne0$, then $\l$ is said to have depth $n$. It is $G$-regular if and only if for each $\a^{\v}\in \Phi^{\v}(G,T)$, viewed as a cocharacter of $T_{0}$, $\j{d\a^{\v},\l_{i}}\ne0$ for some $i\in\{-n,\ldots, 0\}$.
\end{exam}

\begin{exam}[Epipelagic polar data]\label{ex:epi}
Assume for simplicity $G$ is split. Let $T\subset G$ be a maximal torus corresponding to a homomorphism $w: \hZ'(1)\to \WW$ up to conjugacy. We use notations from Example \ref{ex:max tori in split G}, in particular, $w$ factors through $\ov w: \mu_m\incl \WW$.
We call $\ov w$  {\em regular} if the eigenspace $\frt^*_0(1)$ (see \eqref{t0 eigen}, where $\frt_0=\xcoch(\TT)\ot k$) contains a vector $v$ that is not annihilated by the differential of any coroot of $\TT$. By our assumption on $\ch(k)$, this notion is equivalent to saying that $\ov w(\z)$ is a regular element of order $m$ in $\WW$ in the sense of Springer \cite{Spr} for some/any primitive $\z\in \mu_m(k)$.
Under the isomorphism \eqref{colattice}, we take 
\begin{equation*}
\l=v\frac{d\t}{\t^{2}}\in \frt^{*}/\frt^*_{tn}.
\end{equation*}
Then $\l$ is $G$-regular since $v$ is regular semisimple. We call the toral polar datum $(T,\l)$ thus defined an {\em epipelagic polar datum}, due to their close relationship with the epipelagic representations of $p$-adic groups defined by Reeder and Yu \cite{RY}.

  
\end{exam}

\begin{exam}[Homogeneous polar data]\label{ex:homo} Continuing with the setup in Example \ref{ex:epi}, but more generally take any $i\ge1$ that is coprime to $m$, and let 
\begin{equation*}
\l=v \frac{d\t}{\t^{i+1}}
\end{equation*}
where $v\in \frt_{0}^{*}(i)$ is not annihilated by any coroot. Then $(T,\l)$ is again a toral polar datum. We call such $(T,\l)$ {\em homogeneous polar data}, and they are closely related to the Hitchin moduli spaces attached to homogeneous affine Springer fibers studied in \cite{BBAMY}.
\end{exam}

\begin{remark}
Toral polar data $(T,\l)$ for a fixed maximal torus $T$ of $G$ can be thought of a kind of Langlands parameters valued in $T^\vee$, the dual torus of $T$.   

Indeed, when $T$ is split, $\l$ can be viewed as the polar part of a $T^\vee$-connection on the formal disk $\Spec F$. Thus $\l$ can be viewed as part of the de Rham version of the Langlands parameter for $T$. For a general torus $T$ that splits over a finite tame extension $E/F$, we can view $\l$ as the polar part of a $\Gal(E/F)$-equivariant $T^\vee$-connection on the covering disk $\Spec E$.

To relate $\l$ to the usual Langlands parameters for tori, we temporarily work with a finite residue field $k=\FF_q$ (and $F=k\lr{t}$ as usual). We assume $T=T_0\times_{\Spec k}\Spec F$ is split. The datum of $\l\in \frt^*/\frt^*_{tn}$ is the same as a continuous $k$-linear map $\l: \frt_0(\cO_F)\to k$. Composing with a fixed  nontrivial character $\psi: k\to \Qlbar^\times$, we get a continuous homomorphism $\psi_\l: \frt_0(\cO_F)\to \Qlbar^\times$. Let $\frt_0(\cO_F)_{\ge n}=\ker(\frt_0(\cO_F)\to \frt_0(\cO_F/(t^n)))$, $T(\cO_F)_{\ge n}=\ker(T_0(\cO_F)\to T_0(\cO_F/(t^n)))$. We can use the logarithm map to identify $\frt_0(\cO_F)_{\ge 1}/\frt_0(\cO_F)_{\ge p}$ and the $T(\cO_F)_{\ge 1}/T(\cO_F)_{\ge p}$. Thus when the order of pole of $\l$ is $<p$, $\l: \frt_0(\cO_F)\to k$ is trivial on $\frt_0(\cO_F)_{\ge p}$, thus $\l|_{\frt_0(\cO_F)_{\ge1}}$ gives rise to a continuous homomorphism $\l': T(\cO_F)_{\ge 1}\surj T(\cO_F)_{\ge 1}/T(\cO_F)_{\ge p}\to k$. Composing with $\psi$ we get a continuous homomorphism $\phi_{\l}: T(\cO_F)_{\ge 1}\to \Qlbar^\times$. By class field theory we can turn $\phi_{\l}$ into a continuous homomorphism $\rho_\l: P_F\to T^\vee(\Qlbar)$, where $P_F\subset \Gal(F^s/F)$ is the wild inertia. To summarize, when $T$ is split and the order of pole of $\l$ is less than $p$, the irregular polar part of $\l$ is the same datum as a wild Langlands parameter for $T$ in the usual sense.
\end{remark}

Given a toral polar datum $(T,\l)$, let
\begin{equation*}
Y_{(T,\l)}\subset L\frc^{*}
\end{equation*}
be the image of the natural map
\begin{equation*}
\chi^{*}: \l+ (L\frt^{*})_{tn}\to L\frc^{*}. 
\end{equation*}
Here $\l+ (L\frt^{*})_{tn}$ is the $\l$-translation of the closed subscheme $(L\frt^{*})_{tn}$ of the ind-scheme $L\frt^{*}$. In fact  $Y_{(T,\l)}$ can be equipped with a natural structure of a reduced finitely-presented closed subscheme of $L\frc^*$. This will be stated more generally as Lemma \ref{l:Y scheme}, whose proof will be postponed to the sequel.

It is clear from the construction that $Y_{(T,\l)}$ depends only on the $G(F)$-orbit of $(T,\l)$.

\begin{defn} A {\em toral polar characteristic subset}  is a reduced closed subscheme of $L\frc^{*}$ of the form $Y_{(T,\l)}$ for some toral polar datum $(T,\l)$. A {\em toral polar subset of $L\frg^{*}$} is the preimage of a toral polar characteristic subset under the map $\chi^{*}$. We denote 
\begin{equation*}
Z_{(T,\l)}=\chi^{*,-1}(Y_{(T,\l)}).
\end{equation*}
Again $Z_{(T,\l)}$ depends only on the $G(F)$-orbit of $(T,\l)$.
\end{defn}

\subsection{Twisted Levi subgroups}

A {\em twisted Levi subgroup} of $G$ is the centralizer of a torus $A\subset G$; equivalently, $M$ is a subgroup of $G$ that becomes a Levi subgroup of $G$ over $F^s$. Under our assumption \eqref{char ass} on the characteristic of $k$, all twisted Levi subgroups of $G$ will split over a tame extension of $F$ (because any maxima torus $T$ of $M$ is also a maximal torus of $G$, hence splits over a tame extension of $F$, see \S\ref{sss: max tori}). Moreover, the assumption \eqref{char ass} also holds for $M$, because $|\WW_M|$ divides $|\WW|$.

Let $M\subset G$ be a twisted Levi subgroup with Lie algebra $\fm$. Let $A_{M}$ be the maximal torus in the center of $M$. Let $T_{M}=M/M_{\der}$ be the abelianization of $M$, then the natural map $A_{M}\to T_{M}$ is an isogeny of $F$-tori. 


\sss{Relative roots and coroots}\label{sss:rel coroots}
Suppose $M$ is a Levi subgroup of $G$, i.e., $A_M$ and $T_M$ are split tori. We have the set of relative roots $\Phi(G,A_M)\subset \xch(A_M)$ of $G$ with respect to $A_M$: the  weights  of the adjoint action of $A_M$ on $\frg/\fm$. 

Similarly, we can define relative coroots $\Phi^\v(G,T_M)\subset \xcoch(T_M)$ as follows. Choose any maximal torus $T\subset M$. We have a canonical map $T\to T_M$. We define $\Phi^\v(G,T_M)$ to be the nonzero elements in the image of $\Phi^\v(G_{F_\infty},T_{F_\infty})\subset \xcoch(T_{F_\infty})$ under the canonical projection $\xcoch(T_{F_\infty})\to \xcoch(T_{M,F_\infty})=\xcoch(T_M)$.  It is easy to check that the resulting subset $\Phi^\v(G,T_M)\subset \xcoch(T_M)$ is independent of the choices of $T$.

More generally, if $M$ is a twisted Levi subgroup of $G$, so $A_M$ and $T_M$ may not split, the sets $\Phi(G_{F_\infty},A_{M,F_\infty})$ and $\Phi^\v(G_{F_\infty},T_{M,F_\infty})$ are defined by the above procedure, and they carry continuous actions of $\Gal(F_\infty/F)$. Therefore they define finite \'etale $F$-schemes $\un\Phi(G,A_M)$ and $\un\Phi^\v(G,T_M)$.


\sss{Twisted Levi and sub root systems}

Let $T$ be a maximal torus of $G$. Recall the finite \'etale $F$-scheme of roots $\un\Phi(G,T)$ introduced in \S\ref{sss:rs}. A $\QQ$-closed subscheme $\un\Phi'\subset \un\Phi(G,T)$ is a subscheme whose $F_\infty$-points forms a $\QQ$-closed subsystem of $\Phi(G_{F_\infty}, T_{F_\infty})$ in the usual sense.

For any twisted Levi $M\subset G$ containing $T$, we get a $\QQ$-closed subscheme $\un\Phi(M,T)\subset \un\Phi(G,T)$ of roots and a $\QQ$-closed subscheme $\un\Phi^\v(M,T)\subset \un\Phi^\v(G,T)$ of coroots. This construction gives an order-preserving bijection
\begin{eqnarray*}
\{\mbox{Twisted Levi $M\subset G$ such that $T\subset M$}\}&\bij& \{\mbox{$\QQ$-closed subschemes of $\un\Phi(G,T)$}\}\\
&\bij& \{\mbox{$\QQ$-closed subschemes of $\un\Phi^\v(G,T)$}\}.
\end{eqnarray*}

More generally, starting with a twisted Levi subgroup $M$ of $G$, any twisted Levi $M'$ containing $M$ gives a $\QQ$-closed subscheme $\un\Phi(M',A_{M})\subset \un\Phi(G,A_M)$ of relative roots and a $\QQ$-closed subscheme $\un\Phi^\v(M',T_{M})\subset \un\Phi^\v(G,T_M)$ of relative coroots. This construction gives an order-preserving bijection
\begin{eqnarray}\label{tw Levi RS}
\{\mbox{Twisted Levi $M'\subset G$ such that $M\subset M'$}\}&\bij& \{\mbox{$\QQ$-closed subschemes of $\un\Phi(G,A_M)$}\}\\
\notag &\bij&  \{\mbox{$\QQ$-closed subschemes of $\un\Phi^\v(G,T_M)$}\}.
\end{eqnarray}

\sss{Complement} Given a twisted Levi subgroup $M\subset G$, we may identify $\fm^*$ as a subspace of $\frg^*$ canonically: it is the fixed part of the co-adjoint action of $A_M$ on $\frg^*$. We therefore have a decomposition
\begin{equation}\label{decomp g*}
    \frg^*=\fm^*\op (\frg/\fm)^*.
\end{equation}
Dualizing we get a canonical decomposition
\begin{equation}\label{decomp g}
    \frg=\fm\op (\frg/\fm).
\end{equation}

\subsection{General polar data}

Let $M\subset G$ be a twisted Levi subgroup with $T_{M}=M/M_{\der}$. We may apply the discussion of \S\ref{ss:tori} to $T_{M}$ and consider the torsion $\cO_{F}$-module 
\begin{equation*}
\frt_{M}^{*}/\frt^*_{M,tn}.
\end{equation*}
Using \eqref{decomp g*}, we have a canonical embedding
\begin{equation*}
\frt^{*}_{M}\incl \fm^*\incl \frg^{*}.
\end{equation*}

Let $E/F$ be a finite tame extension that splits $T_M$. For a relative coroot $\a^{\v}\in \Phi^{\v}(G_{E},T_{M,E})$ as defined in \S\ref{sss:rel coroots}, its differential $d\a^\v$ gives an element in $\frt_{M}(\cO_{E})$. Pairing with $d\a^{\v}$ gives an $\cO_{E}$-linear map
\begin{equation*}
d\a^{\v}: \frt^{*}_{M}(E)/\frt^*_{M}(E)_{tn}\to \om(E)/\om(\cO_{E}).
\end{equation*}

\begin{defn} An element $\l\in \frt_{M}^{*}/\frt^*_{M,tn}$ is called {\em $G$-regular} if for some finite tame splitting field $E/F$ of $T_M$, and any relative coroot $\a^{\v}\in \Phi^{\v}(G_E,T_{M,E})$, $d\a^{\v}(\l)\in \om(E)/\om(\cO_{E})$ is nonzero.
\end{defn}
This notion is independent of the splitting field $E$.

The following is the key definition.

\begin{defn} A {\em polar datum} for $G$ is a pair $(M,\l)$ where $M$ is a twisted Levi subgroup of $G$ and $\l\in \frt_{M}^{*}/\frt^*_{M,tn}$ is $G$-regular. We denote by $\PD(G)$  the set of all polar data for $G$.
\end{defn}

The group $G(F)$ acts on $PD(G)$ by conjugation. The stabilizer $\Stab_G(M,\l)$ of a polar datum has a natural structure of an algebraic group over $F$.

\begin{lemma}\label{l:stab polar data}
    The stabilizer $\Stab_G(M,\l)$ of any polar datum $(M,\l)$ is equal to $M$ as an algebraic group over $F$.
\end{lemma}
\begin{proof}
    Let $M'=\Stab_G(M,\l)$. It is clear that $M\subset M'\subset N_G(M)$. To show $M=M'$, we may replace $F$ by a tame extension, hence we may assume $M$ is split. Let $T\subset M$ be a split maximal torus so $T=T_0\times_kF$ for some $k$-torus $T_0$. To show $M=M'$ we only need to show their Weyl groups $W(M,T)$ and $W(M',T)$ are equal. 
    
    Let $W=W(G,T)$. We represent $\l$ by the element $\wt\l=\sum_{n>0}\l_nt^{-n}dt\in \frt^*$, where $\l_n\in \frt_0^*$. Then $W(M',T)=W_{\wt \l}$, the stabilizer of $\wt \l$ under $W$, while $W(M,T)$ is the subgroup of $W(M',T)$ generated by root reflections. It is well-known that under our assumption \eqref{char ass} on $\ch(k)$ (in fact a much weaker assumption suffices, see \cite[Lemma 3.3.6]{JY}), $W_{\wt \l}$ is a parabolic subgroup of $W$ for any $\wt \l\in \frt^*$, in particular it is generated by reflections. This shows $W_{\wt\l}=W(M,T)$ and completes the proof.
\end{proof}

Let $\un Y_{(M,\l)}\subset \frc^*(F)=L\frc^{*}(k)$ be the image of 
\begin{equation*}
\l+(L\fm^{*})_{tn}(k)\subset \fm^{*}
\end{equation*}
under the map $\chi^{*}: \fm^{*}\to \frc^{*}$. 

\begin{lemma}\label{l:Y scheme}
    There is a unique reduced fp-closed subscheme  $Y_{(M,\l)}\subset L\frc^*$ whose $k$-points are $\un Y_{(M,\l)}$.
\end{lemma}
Here, fp means finite presentation. Concretely, at least when $G$ is split, we identify $\frc^*$ with an affine space $\AA^r_F$ so that $L\frc^*\cong (L\AA^1)^r$.  We use $L^{\ge i}\AA^1$ to denote the subscheme of $L\AA^1$ involving only powers $t^{\ge i}$, and similarly define the subquotient $L^{[a,b)}\AA^1=L^{\ge a}\AA^1/L^{\ge b}\AA^1$. Then $Y\subset L\frc^*$ being fp-closed means there is a positive integer $N$ and a closed subscheme $Y_N\subset (L^{[-N,N)}\AA^1)^r$ such that $Y\subset(L^{\ge -N}\AA^1)^r$ and it is the preimage of $Y_N$ under the projection $(L^{\ge -N}\AA^1)^r\to (L^{[-N,N)}\AA^1)^r$. 

The proof of this lemma uses results from \cite{BKV}, and will appear in the sequel to this paper.



\begin{defn} A {\em polar characteristic subset} is a reduced closed subscheme of $L\frc^{*}$ of the form $Y_{(M,\l)}$ for some polar datum $(M,\l)$. A {\em polar subset of $L\frg^{*}$} is the preimage of a polar characteristic subset in $L\frc^{*}$ under the map $\chi^{*}$. We denote 
\begin{equation*}
Z_{(M,\l)}=\chi^{*,-1}(Y_{(M,\l)}).
\end{equation*}
\end{defn}
Again $Z_{(M,\l)}$ depends only on the $G(F)$-orbit of $(M,\l)$.

\begin{exam} The pair $(G,0)$ is a polar datum for $G$. We have
\begin{equation*}
Y_{(G,0)}=(L\frc^{*})_{tn}, \quad Z_{(G,0)}=(L\frg^{*})_{tn}.
\end{equation*}
More generally, letting $\frz=\Lie Z(G)$, then for any $\l\in \frz^*/\frz^*_{tn}$, the pair $(G,\l)$ is a polar datum for $G$, with
\begin{equation*}
    Z_{(G,\l)}=\l+(L\frg^{*})_{tn}.
\end{equation*}
\end{exam}

\begin{theorem}\label{th:partition c*} The subschemes $Y_{(M,\l)}$, as $(M,\l)$ runs over $\PD(G)/G(F)$, form a partition of $L\frc^*$.

Consequently, $Z_{(M,\l)}$, where $(M,\l)$ runs over $\PD(G)/G(F)$ form a partition of $L\frg^{*}$.
\end{theorem}
\begin{proof}
    We need to check two things:
    \begin{enumerate}
        \item Any $a\in \frc^*(F)$ lies in $Y_{(M,\l)}$ for some polar datum $(M,\l)$.

        Any point $a\in \frc^*(F)$ lies in the image of $\chi|_\frt^*:\frt^*(F)\to \frc^*(F)$ for some maximal $F$-torus $T$. Let $\wt a\in \frt^*(F)$ be a preimage of $a$, and let $\l\in \frt^*/\frt^*_{tn}$ be the image of $\wt a$. Let $E/F$ be a tame splitting field of $T$, and let $\Phi'\subset \Phi(G_{E}, T_E)$ be those roots such that $d\a^\v(\wt a)\in \om(\cO_E)$. Then $\Phi'$ is a $\QQ$-closed subsystem invariant under $\Gal(E/F)$, which defines a twisted Levi subgroup $M\subset G$ containing $T$ (see \eqref{tw Levi RS}). By construction, $(M,\l)$ is a polar datum, and clearly $a\in Y_{(M,\l)}$.
        
        \item If $a\in\frc^*(F)$ lies in both $Y_{(M_1,\l_1)}$ and $Y_{(M_2,\l_2)}$, then the polar data $(M_1,\l_1)$ and $(M_2,\l_2)$ are $G(F)$-conjugate.

        Suppose $E$ is a tame extension of $F$ and $(M_1,\l_1)$ and $(M_2, \l_2)$ are $G(E)$-conjugate, we claim that they are indeed $G(F)$-conjugate. Indeed, the transporter $Q=\{g\in G|g\cdot (M_1,\l_1)=(M_2,\l_2)\}$ is non-empty since it has an $E$-point, hence it is a torsor under $\Stab_G(M_1,\l_1)=M_1$ (see Lemma \ref{l:stab polar data}). Since $M_1$ is connected and $F$ is a $C_1$-field, $Q$ is a trivial $M_1$-torsor. In particular, $Q(F)\ne \vn$, which implies that  $(M_1,\l_1)$ and $(M_2, \l_2)$ are $G(F)$-conjugate.

        By the above argument, we can replace $F$ by an arbitrary tame extension $E$. We choose $E$ such that $a$ lifts to $\wt a_i\in \frt^*_i$ for some {\em split} maximal $E$-torus $T_i\subset M_i$, $i=1,2$. We need to show that $(M_1,\l_1)$ and $(M_2,\l_2)$ are $G(E)$-conjugate. Since $T_1$ and $T_2$ are  $G(E)$-conjugate, we may as well assume $T_1=T_2$, and we denote it by $T$. Then $\wt a_1,\wt a_2\in \frt^*$ have the same image $a$ in $\frc^*$, therefore there exists $w\in W:=W(G_E,T_E)$ such that $w\wt a_1=\wt a_2$. Choose a lifting $\dot w$ of $w$ to $N_G(T)(E)$ and conjugate $(M_1,\l_1)$ by $\dot w$, we may assume $\wt a_1=\wt a_2$, which we denote by $\wt a$. Now both $\l_1$ and $\l_2$ must be equal to the image of $\wt a$ in $\frt^*(E)/\frt^*(E)_{tn}$, and both $M_1$ and $M_2$ corresponds to the subsystem of $\Phi(G_E, T_E)$ consisting of $\a$ such that $d\a^\v(\l))=0\in \om(E)/\om(\cO_E)$. We conclude $(M_1, \l_1)=(M_2,\l_2)$.
    \end{enumerate}
\end{proof}

\begin{remark}
	The above theorem can be rephrased as a topological Jordan decomposition, which is not restricted to compact elements. It, however, depends on the choice of the uniformizing parameter $t$. The polar partition in Theorem \ref{th:partition c*} is independent of the choice of the uniformizing parameter. 
\end{remark}

\subsection{The case of $\SL_2$} Assume $p\ne 2$ and let $G=\SL_{2}$. In this case, we describe explicitly the partition of $L\frc^{*}$ into polar subsets. 

We identify $\sl_{2}^{*}$ with $\sl_{2}(F)dt$ using the trace pairing $(A,B)=\tr(AB)$. Then $\frc^{*}$ is identified with the affine line of quadratic differentials $\AA^{1}_{F}(dt)^{2}$, with $\chi^{*}: \sl_{2}^{*}\cong\sl_{2}dt\to \AA^{1}(dt)^{2}=\frc^{*}$ given by the determinant map $Adt\mt \det(A)(dt)^{2}$.

We have three kinds of polar data:
\begin{enumerate}
\item Let $S\cong \Gm\subset G$ is a split torus. We have $\frs^{*}=Fdt$, $\frs^*_{tn}=\cO_F dt$, so that $\frs^{*}/\frs^*_{tn}=(F/\cO_{F})dt$. A polar datum $(S,\l)$ is of the form
\begin{equation*}
\l=(a_{-n}t^{-n}+\cdots +a_{-1}t^{-1})dt
\end{equation*}
with $n\ge1$ and $a_{-n}\ne0$. 

The map $\chi^{*}:\frs^{*}\to \frc^*(F)=F(dt)^{2}$ sends $adt$ to $-a^{2}(dt)^{2}$ for $a\in F$. The  corresponding $Y_{(S,\l)}\subset L\frc^{*}=L\AA^{1}$ has the form
\begin{equation*}
(\a_{-2n}t^{-2n}+\a_{-2n+1}t^{-2n+1}+\cdots+\a_{-n-1}t^{-n-1}+t^{-n}\cO_{F})(dt)^{2}
\end{equation*}
where $(\a_{-2n},\cdots, \a_{-n-1})$ are determined by $(a_{-n},\cdots, a_{-1})$ and vice versa. In particular $\a_{-2n}=-a_{-n}^2\ne0$.

As $\l$ varies, the union of $Y_{(S,\l)}$ exhaust points in $a(dt)^2\in \frc^*(F)$ where the valuation $v(a)$ is even and $\le -2$. 

\item Let $T\subset G$ be a nonsplit torus. Let $E=k\lr{t^{1/2}}$, then $T\cong \ker(\Nm: R_{E/F}\Gm\to \Gm)$. Therefore we can identify $\frt$ with $\ker(\tr: E\to F)=t^{1/2}F$, and $\frt_c=\frt\cap \cO_E=t^{1/2}\cO_{F}$. Under the residue pairing we identify $\frt^{*}$ with $t^{-1/2}Fdt$, and $\frt^*_{tn}=t^{-1/2}\cO_{F}dt$, so that $\frt^{*}/\frt^*_{tn}\cong (F/\cO_{F})t^{-1/2}dt$. A polar datum $(T, \l)$ is of the form
\begin{equation*}
\l=(a_{-n}t^{-n}+\cdots+a_{-1}t^{-1})t^{-1/2}dt
\end{equation*}
with $n\ge1$ and $a_{-n}\ne0$. 

The map $\chi^{*}: \frt^{*}\to \frc^{*}(F)=F(dt)^{2}$ sends $t^{-1/2}adt$ to $\Nm(t^{-1/2}a)(dt)^{2}=-t^{-1}a^{2}(dt)^{2}$, for $a\in F$. The corresponding $Y_{(T,\l)}\subset L\frc^{*}=L\AA^{1}$ has the form
\begin{equation*}
(\a_{-2n-1}t^{-2n-1}+\a_{-2n}t^{-2n}+\cdots+\a_{-n-2}t^{-n-2}+t^{-n-1}\cO_{F})(dt)^{2}
\end{equation*}
where $(\a_{-2n-1},\cdots, \a_{-n-2})$ are determined by $(a_{-n},\cdots, a_{-1})$ and vice versa. In particular $\a_{-2n-1}=-a_{-n}^2\ne0$.

As $\l$ varies, the union of $Y_{(T,\l)}$ exhaust points in $a(dt)^2\in \frc^*(F)$ where the valuation $v(a)$ is odd and $\le -3$. 

\item $(G,0)$. The corresponding polar subset $Y_{(G,0)}=(L\frc^*)_{tn}\subset L\frc^{*}$ is the union of $\chi^{*}(\frs^*_{tn})$ and $\chi^{*}(\frt^*_{tn})$, where $S$ and $T$ are the split and nonsplit tori in $G$. From the descriptions of $\frs^*_{tn}$ and $\frt^*_{tn}$ above, we have $\chi^{*}(\frs^*_{tn})=(\cO_{F})^{2}(dt)^{2}$, and $\chi^{*}(\frt^*_{tn})=t^{-1}(\cO_{F})^{2}(dt)^{2}$. Their union is 
\begin{equation*}
Y_{(G,0)}=t^{-1}\cO_{F}(dt)^{2}=\{a(dt)^2|v(a)\ge-1\}.
\end{equation*}
\end{enumerate}
From the above calculations we see immediately that $\frc^{*}(F)=F(dt)^2$ is the disjoint union of $Y_{(M,\l)}$ for various polar data $(M,\l)$.

\subsection{Polar data and Yu data}\label{ss:Yu data}

In \cite{Yu}, J.K.Yu constructs a supercuspidal representation of a $p$-adic group from the following {\em Yu datum}:
\begin{equation}\label{Yu data}
    (\orr{G}, \pi_0, \orr{\phi})
\end{equation}
where $\orr{G}=(G^{(0)}\subsetneq G^{(1)}\subsetneq \cdots\subsetneq G^{(d)}=G)$ is a sequence of twisted Levi subgroups of $G$, $\pi_0$ an irreducible supercuspidal depth zero representation of $G^{(0)}$, and $\orr{\phi}=(\phi_0,\phi_1,\cdots, \phi_d)$, where each $\phi_i$ is a linear character of $G^i(F)$. In this section, starting from a polar datum, we shall obtain something very close to a Yu datum, except for $\pi_0$, whose Lie algebra analog will be recaptured in \S\ref{ss:cs loop}.

\sss{From polar data to twisted Levi sequences}\label{sss:Levi seq} 
Given a polar datum $(M,\l)$ for $G$, we will attach such a sequence of twisted Levi sequence as follows. 

Recall the relative coroot scheme $\un\Phi^{\v}(G,T_{M})$ from \S\ref{sss:rel coroots}. For any $r\in \QQ_{\ge 0}$, let 
\begin{equation*}
\un\Phi^{\v}(G,T_{M})_{\ge -r}\subset \un\Phi^{\v}(G,T_{M})
\end{equation*}
be the subscheme whose $F_{\infty}$-points are those $\a^{\v}:\GG_{m,F_{\infty}}\to T_{M,F_{\infty}}$ such that $\j{d\a^{\v},\l}\in t^{-r}\cO_{F_{\infty}}\frac{dt}{t}$. 


For $r>0$, let
\begin{equation*}
\un\Phi^\v(G,T_{M})_{>-r}=\bigcup_{r'<r}\un\Phi^\v(G,T_{M})_{\ge-r'}.
\end{equation*}
We also define $\un\Phi^\v(G,T_{M})_{>0}=\vn$. 

There are finitely many values $r\ge0$ such that $\un\Phi^\v(G,T_{M})_{\ge-r}\ne \un\Phi^\v(G,T_{M})_{>-r}$. We order them as
\begin{equation}\label{r seq}
\orr{r}: 0\le  r_{0}<r_{1}<\cdots<r_{d-1}.
\end{equation}

By the bijection \eqref{tw Levi RS}, each $\un\Phi(G,A_{M})_{\ge -r_{j}}$ corresponds uniquely to a twisted Levi subgroup $G^{(j)}$ of $G$ that contains $M$. Note that $G^{(0)}=M$. We set $G^{(d)}=G$, and get a twisted Levi sequence $\orr{G}$:
\begin{equation}\label{tw Levi seq}
\orr{G}=(M=G^{(0)}\subsetneq G^{(1)}\subsetneq \cdots\subsetneq G^{(d)}=G).
\end{equation}



\sss{Linear characters} We now supply the analog of the sequence of characters $\orr{\phi}$ in the Yu data from $\l$.

Base change $T_M$ to $F_\infty$ where it splits, we have a canonical isomorphism $T_{M,F_\infty}\cong \xcoch(T_{M,F_\infty})\ot_{\ZZ}\Gm$, hence a canonical isomorphism of $F_\infty$-vector spaces
\begin{equation}\label{trivialize frtM}
    \frt^*_M(F_\infty)\cong \xch(T_{M,F_\infty})\ot_{\ZZ}\om(F_\infty).
\end{equation}
Using rational powers of $t$, we can write $\l\in \frt^*_M(F_\infty)/\frt^*_{M}(\cO_\infty)_{tn}$ uniquely as
\begin{equation}\label{decomp lambda}   
\l=\l^{(d)}+\l^{(d-1)}+\cdots+\l^{(0)}\mod \frt^*_{M}(\cO_\infty)_{tn}.
\end{equation}
such that under the isomorphism \eqref{trivialize frtM}, $\l^{(j)}$, for $j>0$, only has powers $t^{-r}\frac{dt}{t}$ for $r\in (r_{j-1}, r_j]$ (convention: $r_d=\infty$); and when $j=0$, $\l^{(0)}$ only involves $\frac{dt}{t}$. The uniqueness of the decomposition shows that each $\l^{(j)}$ is in fact inside $\frt_M^*$. 
From the construction we see that $G^{(j)}$ is the simultaneous centralizer
\begin{equation*}
    G^{(j)}=C_G(\l^{(d)},\l^{(d-1)},\cdots, \l^{(j)}).
\end{equation*}
In particular, $\l^{(j)}$ can be viewed as a Lie algebra homomorphism
\begin{equation*}
    \l^{(j)}: \frg^{(j)}\to k.
\end{equation*}
The sequence of linear characters $\orr{\l}=(\l^{(0)}, \l^{(1)},\cdots, \l^{(d)})$ is the Lie algebra analog of $\orr{\phi}$ in the Yu data.

\begin{remark}
    The passage from polar data to Yu data depends on the choice of the compatible system of uniformizers $t^{1/n}$ of tame extensions of $F$ to write the decomposition \eqref{decomp lambda}. However, the notion of polar data itself does not depend on such choices.
\end{remark}

\section{Character sheaves on loop Lie algebra}

We will propose a definition of the category of character sheaves on the loop Lie algebra $L\frg$ using Fourier transform, namely we shall first define a category of sheaves on the dual Lie algebra $L\frg^*$. 

The results stated in this section will be proved in a future publication.

Fix a prime $\ell\ne p$. The sheaves we consider will have $\Qlbar$-coefficients. We also formally adjoin a half Tate twist $(1/2)$. Fix a nontrivial character $\psi: \FF_p\to \Qlbar^\times$. Let $\AS_\psi$ be the corresponding Artin-Schreier local system on $\AA^1$ (over $k$). We use $\AS_\psi$ to define the Fourier-Deligne transform for sheaves on vector spaces over $k$.

\subsection{Recollections on character sheaves on Lie algebras}\label{ss:ind cusp}

Our proposal for the definition of character sheaves on the loop Lie algebra will be modeled after Lusztig's definition of character sheaves on a finite-dimensional Lie algebra, which we now review. 

Let $G$ be a connected reductive group over $k$. Lusztig \cite{L-Fourier} has defined the notion of character sheaves on the Lie algebra $\frg$, see also \cite{mirkovic_character_2004}. An object $\cF\in D_{G}(\frg^{*})$ is called an {\em orbital sheaf} if it is a simple perverse sheaf supported on the closure of a single coadjoint orbit. Let
\begin{equation*}
    \OS(\frg^*)\subset D_G(\frg^*)
\end{equation*}
be the cocomplete full subcategory generated by orbital sheaves. 

The category of character sheaves on $\frg$ is defined to be the full subcategory
\begin{equation*}
\CS(\frg)\subset D_{G}(\frg)
\end{equation*}
that is the image of $\OS(\frg^*)$ under the Fourier transform. 

Below we recall Lusztig's classification of orbital and character sheaves.

\sss{Nilpotent orbital sheaves}
Let $\cN^*\subset \frg^*$ be the nilpotent cone. Consider the full subcategory $D_G(\cN^*)\subset \OS(\frg^*)$ of equivariant sheaves supported on $\cN^*$. We denote by
\begin{equation*}
\CS_{\nil}(\frg)\subset \CS(\frg)
\end{equation*}
the image of Fourier transform of $D_G(\cN^*)$. Simple perverse sheaves in $\CS_{\nil}(\frg)$ (Fourier transform of nilpotent orbital sheaves) are Lie algebra analogues of unipotent character sheaves on $G$. We call $\CS_{\nil}(\frg)$ the category of {\em nilpotent character sheaves}.


\begin{defn}
    A {\em cuspidal datum} for $\frg^*$ is a triple $\g=(L,\cO, \cL)$ where
\begin{itemize}
    \item $L$ is a Levi factor of a parabolic subgroup of $G$;
    \item $\cO\subset \cN^*_L\subset \frl^*$ is a nilpotent coadjoint orbit of $L$;
    \item $\cL$ is an irreducible cuspidal $L$-equivariant local system on $\cO$. An $L$-equivariant local system $\cL$ on the nilpotent orbit $\cO$ is cuspidal if and only if the Fourier transform of its extension by zero is also supported on a single nilpotent orbit.
\end{itemize}
\end{defn}

Very few nilpotent orbits support cuspidal local systems, and they are classified by Lusztig in \cite{L-IC}. It turns out that if $\cL$ is cuspidal then it is  clean: $i_{!}\cL\isom i_{*}\cL$. We denote the latter by $\wt\cL\in D_L(\frl^*)$.

We denote the set of cuspidal data for $\frg^*$ by $\CD(\frg^*)$. It is acted on $G$ by conjugation.

\sss{Parabolic induction}
For a parabolic subgroup $P\subset G$ with unipotent radical $N_P$ and Levi quotient $L_{P}$, let $(\frn_P)^\bot\subset\frg^*$ be the annihilator of $\frn_P=\Lie N_P$, and $\frl_P=\Lie L_P$.  Consider  maps
\begin{equation*}
\xymatrix{[\frg^*/G] & [\frn_P^\bot/P]\ar[l]_-{\pi}\ar[r]^-{\e} & [\frl^*_{P}/L_{P}].
}
\end{equation*}
The diagram
\begin{equation*}
\xymatrix{\ind_{P}^{G}=\pi_{!}\e^{*}: D_{L_P}(\frl^*_{P})\ar@<1ex>[r] &  D_G(\frg^*): \res_{P}^{G}=\e_{*}\pi^{!} \ar@<1ex>[l]}
\end{equation*}
define the induction and restriction functors, which form a pair of adjoint functors.

By the generalized Springer correspondence \cite{L-CS}, we have a block decomposition
\begin{equation}\label{orb g block}
    D_G(\cN^*)\cong \bigoplus_{\g\in \CD(\frg^*)/G}\OS_{\g}(\frg^*)
\end{equation}
where for $\g=(L,\cO,\cL)\in \CD(\frg^*)$, $\OS_{\g}(\frg^*)\subset D_G(\cN^*)$ is the full subcategory generated under colimits (in particular, direct summands) by the objects
\begin{equation*}
    \cK_{\g, P}=\ind_P^G(\wt\cL)\in D_G(\cN^*)
\end{equation*}
for some (equivalently any) parabolic subgroup $P$ containing $L$ as a Levi factor. 


Taking Fourier transform yields a block decomposition
\begin{equation*}
\CS_{\nil}(\frg)=\bigoplus_{\g\in \CD(\frg^*)/G}\CS_{\g}(\frg)
\end{equation*}
where $\CS_{\g}(\frg)\subset D_G(\frg)$ is the image of $\OS_\g(\frg^*)$ under the Fourier transform.

\sss{General orbital sheaves}

We use the following notion to reformulate the classification of general orbital sheaves. 
\begin{defn}\label{d:ext cusp}
    An {\em extended cuspidal datum} for $\frg^*$ is a tuple $\y=(M_0,\l_0, L,\cO,\cL)$ where
    \begin{itemize}
        \item $\l_0\in \frg^*$ is a semisimple element;
        \item $M_0=C_G(\l_0)$ is the centralizer of $\l_0$;
        \item $(L,\cO,\cL)$ is a cuspidal datum for the dual Lie algebra $\fm_0^*$ of $M_0$.
    \end{itemize}
\end{defn}
Of course $M_0$ is determined by $\l_0$, but it will be convenient to include it in the data.

Let $\ECD(\frg^*)$ be the set of extended cuspidal data for $\frg^*$. Then $G$ acts on $\ECD(\frg^*)$ by conjugation. 

For each $\y=(M_0,\l_0, L,\cO,\cL)\in \ECD(\frg^*)$, we have a closed embedding of stacks
\begin{equation*}
    i_{M_0,\l_0}: [\cN^*_{M_0}/M_0]\xr{t_{\l_0}}[(\l_0+\cN^*_{M_0})/M_0]\incl [\frg^*/G].
\end{equation*}
where the first map is shifting by $\l_0$ (which is invariant under $M_0$). This gives a fully faithful functor
\begin{equation*}
    i_{(M_0,\l_0)!}: D_{M_0}(\cN^*_{M_0})\incl D_G(\frg^*).
\end{equation*}
We define $\OS_\y(\frg^*)$ to be the image of $\OS_{\g}(\fm^*_0)\subset D_{M_0}(\cN^*_{M_0})$ under the above functor, where $\g=(L,\cO,\cL)\in \CD(\fm^*_0)$. We then have a block decomposition
\begin{equation*}
    \OS(\frg^*)=\bigoplus_{\y\in \ECD(\frg^*)/G}\OS_{\y}(\frg).
\end{equation*}
Taking Fourier transform we obtain a block decomposition for $\CS(\frg)$:
\begin{equation}\label{CS g blocks}
\CS(\frg)=\bigoplus_{\y\in \ECD(\frg^*)/G}\CS_{\y}(\frg).
\end{equation}

For the loop Lie algebra, polar data will play the role of the pair $(M_0,\l_0)$ above. We will define in \S\ref{ss:cs loop} an affine analog of the extended cuspidal data, and use it to define character sheaves on the loop Lie algebra.

\subsection{Fourier transform on loop Lie algebra}

We briefly sketch how Fourier transform can be defined for sheaves on infinite-dimensional vector spaces such as $L\frg$, and how it can be upgraded to a $LG$-equivariant version.

\sss{Categorical half-densities}
We consider a renormalized version $D^\bu(L\frg)$ of $\Qlbar$-sheaves on $L\frg$ that can be thought of as a geometric incarnation of half-densities.

The category $D^\bu(L\frg)$ is defined as the colimit
\begin{equation*}
    D^\bu(L\frg)=\colim_{\L'\subset \L} D(\L/\L')
\end{equation*}
over pairs of lattices $\L'\subset \L\subset \frg$, with the transition functors, for $\L_1'\subset \L'\subset \L\subset \L_1$,
\begin{equation}\label{trans D(L/L')}
i_{*}\pr^{\bu}: D(\L/\L')\to D(\L_{1}/\L_{1}').
\end{equation}
Here $i: \L/\L'_1\incl \L_1/\L'_1$ is the inclusion, $\pr: \L/\L'_1\to \L/\L'$ is the projection, $\pr^{\bu}=p^{*}[d](d/2)\cong p^{!}[-d](-d/2)$, where $d=\dim_{k}(\L'/\L_{1}')$. Note that $i_*\pr^\bu$ is $t$-exact for the perverse t-structures. The category $D^\bu(\frg)$ thus carries a perverse $t$-structure extending the usual perverse $t$-structures on each $D(\L/\L')$.

Similar construction applies to $L\frg^*$.

\sss{Fourier transform} The Fourier transform functor
\begin{equation}\label{FT plain}
    \FT: D^\bu(L\frg^*)\to D^\bu(L\frg)
\end{equation}
is defined as the colimit of the usual Fourier transform
\begin{equation*}
\FT_{\L/\L'}:D(\L/\L')\to D(\L'^\bot/\L^\bot). 
\end{equation*}
Here for a lattice $\L\subset L\frg^*$, $\L^\bot\subset \frg^*$ is the dual lattice defined as the annihilator of $\L$ under the residue pairing \eqref{res pairing g}. Note that $\FT_{\L/\L'}$ is also normalized by a shift and Tate twist to be perverse $t$-exact and preserving pure weight zero objects.

\sss{$LG$-equivariant half-densities}
To define the category of $LG$-equivariant sheaves (or rather half-densities) on $L\frg$, we follow an idea in the work \cite{BKV0}.  Fix an Iwahori subgroup $\bI\subset LG$. Consider the category $\frP$ of standard paraboric subgroups of $LG$: objects in $\frP$ are parahoric subgroups $\bP\subset LG$ containing a fixed Iwahori $\bI$; for $\bP,\bP'\in \frP$, the morphism set $\Mor_{\frP}(\bP,\bP')$ is the set of double cosets $\bP' g\bP$ such that $\Ad(g)\bP\subset \bP'$. We define
\begin{equation}\label{lim LGeq}
    D^\bu_{LG}(L\frg):= \lim_{\frP}D_{\bP}^\bu(L\frg)
\end{equation}
where transition functors are $!$-forgetful functors. Passing to left adjoints, $D^\bu_{LG}(L\frg)$ can also be described as a colimit
\begin{equation}\label{colim LGeq}
    D^\bu_{LG}(L\frg)\cong \colim_{\frP}D_{\bP}^\bu(L\frg)
\end{equation}
using averaging functors as transition functors. The canonical functor $D_{\bP}^\bu(L\frg)\to D^\bu_{LG}(L\frg)$ under the colimit presentation \eqref{colim LGeq} is the $!$-averaging from $\bP$-equivariance to $LG$-equivariance. In particular, for any subgroup $\bK\subset LG$ that is contained in a parahoric subgroup of finite codimension, we have a $!$-averaging functor
\begin{equation*}
    \Av_{\bK,!}^{LG}: D^\bu_{\bK}(L\frg)\to D^\bu_{LG}(L\frg).
\end{equation*}
The same discussion applies to define $D^\bu_{LG}(L\frg^*)$.

For a closed $LG$-invariant sub-indscheme $Z\subset L\frg$, we denote by $D^\bu_{LG}(Z)$ the full subcategory of $D^{\bu}_{LG}(L\frg)$ consisting of objects whose support (in the $!$-sense) is contained in $Z$.

\sss{Equivariant Fourier transform}\label{sss:eq FT}
The definition of the Fourier transform \eqref{FT plain} easily upgrades to $\bP$-equivariant categories: one only needs to consider lattices that are stable under $\bP$. Taking limits as $\bP$ varies, we get the $LG$-equivariant Fourier transform
\begin{equation}\label{LG equiv FT}
    {}^{LG}\FT: D^\bu_{LG}(L\frg^*)\isom D^\bu_{LG}(L\frg).
\end{equation}
The basic calculation of ${}^{LG}\FT$ is as follows. Let $\bK\subset LG$ be a subgroup of finite codimension in a parahoric subgroup. Let $\L\subset \frg$ be a $\bK$-stable lattice, and $\l\in \frg^*/\L^\bot$ be such that the shifted lattice $\l+\L^\bot$ is also $\bK$-stable. Let $\const{\l+\L^\bot:\L^\bot}$ be the image of the constant sheaf in $D_{\bK}(\wt\L/\L^\bot)$ (for any $\bK$-stable lattice $\wt\L$ containing both $\l$ and $\L^\bot$) under the canonical embedding $D_{\bK}(\wt\L/\L^\bot)\incl D^\bu_{\bK}(L\frg^*)$. 

On the other hand, we may view $\l$ as a continuous linear function $\l: \L\to k$ using the residue pairing. For any $\bK$-stable lattice 
$\L'\subset \ker(\l)\subset \L$, let $\AS_{\l, \L/\L'}[d_{\L/\L'}](d_{\L/\L'}/2)\in D_{\bK}^\bu(\L/\L')$ be the pulling back of the Artin-Schreier local system $\AS_\psi$ along $\l$. Here $d_{\L/\L'}=\dim_k(\L/\L')$. Let $\AS_{\l,\L}\in D^\bu_{\bK}(L\frg)$ be the image of and such $\AS_{\l, \L/\L'}[d_{\L/\L'}](d_{\L/\L'}/2)$. Then we have
\begin{equation*}
    {}^{LG}\FT(\Av^{LG}_{\bK, !}\const{\l+\L^\bot:\L^\bot})\cong \Av^{LG}_{\bK, !}\AS_{\l, \L}\in D^\bu_{LG}(L\frg).
\end{equation*}

\subsection{Character sheaves on $L\frg$}\label{ss:cs loop}

For the loop group $LG$, a {\em parahoric-Levi subgroup} is a subgroup $L\subset LG$ such that there exists a parahoric subgroup $\bP\subset LG$ containing $L$ as a Levi factor, i.e., the composition $L\subset \bP\to L_\bP$ (the reductive quotient of $\bP$) is an isomorphism. 

\begin{defn}
    A {\em polar-cuspidal datum} for $G$ is a tuple $\xi=(M,\l, L, \cO, \cL)$ where
    \begin{itemize}
        \item $(M,\l)$ is a polar datum for $G$;
        \item $L$ is a parahoric-Levi  subgroup of $LM$;
        \item $\cO\subset \cN^*_{L}$ is a nilpotent co-adjoint orbit of $L$, and $\cL$ is an $L$-equivariant cuspidal local system on $\cO$.
    \end{itemize}
\end{defn}
Let $\PCD(G)$ be the set of polar-cuspidal data for $G$. There is an action of $G(F)$ on $\PCD(G)$ by conjugation, lifting its action on polar data. 

\begin{remark}\label{r: toral no cusp}
    For a toral polar datum $(T,\l)$, there is only one choice of $(L,\cO,\cL)$ making $(T,\l, L,\cO,\cL)$ a polar-cuspidal datum: namely, $L$ is the Levi factor of the unique parahoric subgroup $\bT$ of $LT$, $\cO=\{0\}$ and $\cL$ the trivial $L$-equivariant local system. 
\end{remark}

\sss{The complexes $\cK_{\xi,\bQ}$}\label{sss:Kxi Q}
Fix a polar-cuspidal datum $\xi=(M,\l,L,\cO,\cL)$, we shall define a full subcategory
\begin{equation*}
    \OS_\xi(L\frg^*)\subset D^\bu_{LG}(L\frg^*)
\end{equation*}
as follows. Choose a parahoric subgroup $\bQ\subset LM$ containing $L$ as a Levi factor. Let $\frq^\bot\subset L\frg^*$ be the annihilator of $\frq=\Lie \bQ$ under the residue pairing. Consider the following subscheme of $L\fm^*$ stable under the coadjoint action of $\bQ$
\begin{equation*}
    \l+\cO+\frq^\bot\subset \l+(L\fm^*)_{tn}.
\end{equation*}
Since $\frt^*_{M,tn}\subset \frq^\bot$, and $\l$ is well-defined up to $\frt^*_{M,tn}$, the above coset makes sense.  Consider the diagram where the maps are induced by the obvious inclusions or projections
\begin{equation*}
   \xymatrix{ [\cO/L] & \ar[l]_-{\e_{\bQ, \l, \cO}}[(\l+\cO+\frq^\bot)/\bQ] \ar[r]^-{\pi_{\bQ, \l, \cO}} & [\l+(L\fm^*)_{tn}/LM]
   }
\end{equation*}
We form the object
\begin{equation*}
    \cK_{\xi,\bQ}^M:=(\pi_{\bQ,\l,\cO})_!\e_{\bQ, \l,\cO}^*\cL\in D^\bu_{LM}(\l+(L\fm^*)_{tn})\subset D^\bu_{LM}(L\fm^*).
\end{equation*}

We observe that the inclusion
\begin{equation*}
\l+(L\fm^{*})_{tn}\subset Z_{(M,\l)}
\end{equation*}
induces an isomorphism of stacks
\begin{equation}\label{Z reduce to M}
[(\l+(L\fm^{*})_{tn})/LM]\isom [Z_{(M,\l)}/LG].
\end{equation}
We mention without proof that this isomorphism induces an equivalence of categories
\begin{equation*}
    D^\bu_{LM}(\l+(L\fm^*)_{tn})\isom D^\bu_{LG}(Z_{(M,\l)}).
\end{equation*}
This is not obvious from our definition of $D^\bu_{LM}(L\fm^*)$ and $D^\bu_{LG}(L\frg^*)$ either as a limit \eqref{lim LGeq} or as a colimit \eqref{colim LGeq}. It can be proved by using Theorem \ref{th:compare JKY} or Remark \ref{r:Kr}.

Therefore it makes sense to view $\cK_{\xi,\bQ}^M$ as an object in $D^\bu_{LG}(Z_{(M,\l)})$. Via the embedding $i_{(M,\l)}: Z_{(M,\l)}\incl L\frg^*$, we define
\begin{equation}\label{ind cusp LM}
        \cK_{\xi,\bQ}:=i_{(M,\l)!}\cK^M_{\xi,\bQ}\in D^\bu_{LG}(L\frg^*).
\end{equation}

\begin{defn}
    Let $\OS_{\xi}(L\frg^*)\subset D^\bu_{LG}(L\frg^*)$ be full subcategory generated under colimits (in particular direct summands) by the objects $\cK_{\xi,\bQ}$ for various parahoric subgroups $\bQ\subset LM$ containing $L$ as a Levi factor. 
\end{defn}

\begin{remark}
    It can be checked that $\OS_{\xi}(L\frg^*)$ is generated under colimits by an objet of the form \eqref{ind cusp LM} for a single choice of parahoric subgroup $\bQ\subset LM$.
\end{remark}

\begin{exam}[Nilpotent case]
Consider the case $(M,\l)=(G,0)$. In this case, let $\xi=(G,0,L,\cO,\cL)$ be a polar-cuspidal datum.  For any parahoric subgroup $\bP\subset LG$ with Lie algebra $\frp$, pro-nilpotent radical $\frp^+$ and Levi quotient $L_{\bP}$, consider the induction diagram
\begin{equation*}
\xymatrix{[L\frg^{*}/LG] & [(\frp^+)^{\bot}/\bP]\ar[l]_-{\pi}\ar[r]^-{\e} & [\frl^{*}_{\bP}/L_{\bP}].
}
\end{equation*}
Consider the parahoric induction functor
\begin{equation}\label{ind Lg*}
\ind_{\bP}^{LG}=\pi_{!}\e^{*}: D_{L_{\bP}}(\frl^{*}_{\bP})\to D^\bu_{LG}(L\frg^{*}).
\end{equation}
The functor $\pi_!$ is defined because $\pi$ is ind-proper. If $\bP$ contains $L$ as a Levi factor, we may view the clean extension $\wt\cL$ of $\cL$ as an object in $D_{L_\bP}(\frl^*_\bP)\cong D_{L}(\frl^*)$. Then $\OS_\xi(L\frg^*)$ is generated by $\ind_{\bP}^{LG}(\wt\cL)$ for some (equivariantly any) parahoric subgroup $\bP\subset LG$ containing $L$ as a Levi factor.

If we further specialize to the case $\bP=\bI$, so that $L=T_0\subset \bI$ is a maximal torus, we call the resulting polar-cuspidal datum the {\em principal polar-cuspidal datum}
\begin{equation}\label{princ pcd}
    \xi_0=(G,0,T_0, \{0\}, \Qlbar).
\end{equation}
The sheaf $\cK_{\xi_0, \bI}$ is the affine analogue of the Springer sheaf (direct image of the Springer resolution $\wt\cN\to \cN$).
\end{exam}

\begin{exam}[Toral case]\label{ex:toral coCS}
    Let $(T,\l)$ be a toral polar datum. As we mentioned in Remark \ref{r: toral no cusp}, there is a unique polar-cuspidal datum $\xi$ extending $(T,\l)$. Let $\bT$ be the unique parahoric subgroup of $LT$, and let 
    \begin{equation*}
    \pi_{\bT, \l}: [\l+(L\frt^*)_{tn}/\bT]\to [L\frg^*/LG]
    \end{equation*}
    be induced by the natural inclusions. Then $\OS_\xi(L\frg^*)$ is generated by the object
    \begin{equation}\label{Kxi toral}
        \cK_\xi=(\pi_{\bT,\l})_!\Qlbar.
    \end{equation}
The sheaf  $\cK_{\xi}$ is supported on $Z_{(T,\l)}$ but its stalks are isomorphic to the group algebra $\Qlbar[\pi_0(LT)]$. This is because $\pi_{\bT,\l}$ factors as the composition
\begin{equation*}
    [\l+(L\frt^*)_{tn}/\bT]\to  [\l+(L\frt^*)_{tn}/LT]\cong [Z_{(T,\l)}/LG]\incl [L\frg^*/LG].
\end{equation*}
where the first map is a $\pi_0(LT)$-torsor, and the last one is a closed embedding.
\end{exam}

Finally we define character sheaves on $L\frg$.
\begin{defn}
\begin{enumerate}
    \item For a polar-cuspidal datum $\xi$ of $G$, define the full subcategory
    \begin{equation*}
        \CS_\xi(L\frg)\subset D^\bu_{LG}(L\frg)
    \end{equation*}
    to be the image of $\OS_{\xi}(L\frg^*)$ under the $LG$-equivariant Fourier transform \eqref{LG equiv FT}.

    \item Define $\CS(L\frg)\subset D^\bu_{LG}(L\frg)$ to be the cocomplete full subcategory generated by $\CS_\xi(L\frg)$ for all $\xi\in \PCD(G)$.  
\end{enumerate}
\end{defn}

Recall the complexes $\cK_{\xi,\bQ}$ in \S\ref{sss:Kxi Q}. Denote its Fourier transform by
\begin{equation}\label{Sxi Q}
    \cS_{\xi, \bQ}:={}^{LG}\FT(\cK_{\xi,\bQ})\in \CS_\xi(L\frg).
\end{equation}
Then $\CS_\xi(L\frg)$ is generated under colimits by $\cS_{\xi,\bQ}$ for some (equivalently, any) parahoric $\bQ\subset LM$ containing $L$ as a Levi factor. In the toral case where $\bQ$ is unique, we denote \eqref{Sxi Q} by $\cS_\xi$, i.e.,
\begin{equation}\label{Sxi toral}
    \cS_\xi={}^{LG}\FT(\cK_{\xi}), \quad \xi=(T,\l, \cdots) \textup{ is toral}.
\end{equation}

\begin{exam}[Affine Grothendieck-Springer sheaf]\label{ex:aff Spr} 
For the principal polar-cuspidal datum $\xi_0$ in \eqref{princ pcd}, we have
\begin{equation*}
    \cS_{\xi_0, \bI}\cong \Av_{\bI,!}^{LG}\const{\Lie\bI}.
\end{equation*}
The right hand side is the $!$-direct image of constant sheaf along the affine version of the Grothendieck alteration
\begin{equation*}
    \pi_{\bI}: [(\Lie \bI)/\bI]\to [L\frg/LG].
\end{equation*}
Therefore $\cS_{\xi_0,\bI}$ is the affine Grothendieck-Springer sheaf studied in \cite{BKV}. The principal block $\CS_{\xi_0}(L\frg)$ is generated by $\cS_{\xi_0,\bI}$ under colimits (including taking direct summands).
\end{exam}

\subsection{Bernstein decomposition}\label{ss:block}

\sss{Descent to twisted Levi}\label{sss:reduce to M}

For a polar-cuspidal datum $\xi=(M,\l,\cdots)$ of $G$,  the isomorphism \eqref{Z reduce to M} implies  equivalences
\begin{equation}\label{CS reduce to M}
    \OS_\xi(L\frg^*)\cong \OS_{\g}(L\fm^*), \quad \CS_\xi(L\frg)\cong \CS_{\g}(L\fm)
\end{equation}
where $\g$ is the polar-cuspidal datum $(M,0,L,\cO,\cL)$ of the twisted Levi subgroup $M$.

The following result may be thought of as the Lie algebra analog of the Bernstein decomposition.

\begin{theorem}\label{th:block} We have a block decomposition
\begin{equation*}
\CS(L\frg)=\bigoplus_{\xi\in \PCD(G)/G(F)}\CS_{\xi}(L\frg)
\end{equation*}
where the sum is over all $G(F)$-conjugacy classes of polar-cuspidal data.
\end{theorem}
\begin{proof}[Sketch of proof]
    Since Fourier transform is an equivalence, it suffices to prove that for $\xi=(M,\l,L,\cO,\cL)$ and $\xi'=(M',\l',L',\cO',\cL')$ that are not $G(F)$-conjugate, $\OS_\xi(L\frg^*)$ and $\OS_{\xi'}(L\frg^*)$ are orthogonal to each other. 
    
    If the polar data $(M,\l)$ and $(M',\l')$ are not $G(F)$-conjugate, then $Z_{(M,\l)}/LG$ and $Z_{(M',\l')}/LG$ are disjoint closed substacks of $L\frg^*/LG$ by Theorem \ref{th:partition c*}. In this case,  objects in $\OS_\xi(L\frg^*)$ and $\OS_{\xi'}(L\frg^*)$ have disjoint supports, hence they are orthogonal to each other.

    When $(M,\l)=(M',\l')$, let $\g=(M,0,L,\cO,\cL)$ and $\g'=(M,0,L',\cO',\cL')$ be the corresponding nilpotent polar-cuspidal data for $M$. By \eqref{CS reduce to M}, we reduce to showing the orthogonality between $\OS_{\g}(L\fm^*)$ and $\OS_{\g'}(L\fm^*)$ within $D^\bu_{LM}(L\fm^*)$.  This can be deduced from the finite-dimensional block decomposition \eqref{orb g block}.
\end{proof}


\subsection{Comparison with J-K.Yu's construction}\label{ss:compare Yu}

In this section we state a close relationship between certain blocks of $\CS(L\frg)$ and the characters of supercuspidal representations of a $p$-adic group constructed by J.K.Yu.

Below we fix a polar datum $(M,\l)$ for $G$. In this subsection, we always assume $r_0>0$ with $r_0$ as in the sequence \eqref{r seq} attached to $\l$.

We will define a subgroup $\bJ\subset LG$ and a linear character on its Lie algebra using $(M,\l)$ that is similar to the compact open subgroup and the character on it used by J-K.Yu in his construction of supercuspidal representations of $p$-adic groups by compact induction.

\sss{Decomposition of $\frg$}

For each $G^{(j)}$ with its twisted Levi $G^{(j-1)}$, by \eqref{decomp g*} and \eqref{decomp g} we have canonical decompositions of $\frg^{(j)}$ and $\frg^{(j),*}$ into $F$-subspaces:
\begin{equation*}
    \frg^{(j)}=\frg^{(j-1)}\op V^{(j)}, \quad \frg^{(j), *}=\frg^{(j-1),*}\op V^{(j),*}.
\end{equation*}
Here $V_j(\frg^*)$ is the orthogonal complement of $\frg^{(j-1)}$ inside $\frg^{(j), *}$, and $V_j(\frg)$ is the orthogonal complement of $\frg^{(j-1),*}$ inside $\frg^{(j)}$. Together we get decompositions
\begin{equation}\label{decomp gd}
\frg=\bigoplus_{j=0}^d V^{(j)}, \quad \frg^*=\bigoplus_{j=0}^d V^{(j),*},
\end{equation}
such that
\begin{equation*}
\frg^{(j)}=\bigoplus_{j'=0}^{j} V^{(j')}, \quad \frg^{(j),*}=\bigoplus_{j'=0}^{j} V^{(j'),*}.
\end{equation*}
Each $V^{(j)}$ and $V^{(j),*}$ is stable under the (co-)adjoint action of $G^{(j-1)}$. In particular, all terms in the decompositions \eqref{decomp gd} are stable under the (co-)adjoint action of $M$.

\sss{Moy-Prasad filtration}

Let $\frB(G)$ be the extended building of $G(F)$. As in \cite[Section 2]{Yu}, we fix embeddings
\begin{equation*}
    \frB(M)\incl \frB(G^{(1)})\incl \frB(G^{(2)})\incl \cdots \incl \frB(G).
\end{equation*}
Fix a point $x\in\frB(M)$, viewed as a point in the building of each $G^{(j)}$. For $r\in \QQ$, let 
\begin{equation*}
\frg^{(j)}_{>r} \subset \frg^{(j)}_{\ge r}   
\end{equation*}
be the Moy-Prasad subalgebras of $\frg^{(j)}$ defined using the point $x$, which are $\cO_F$-lattices in $\frg^{(j)}$. For the Moy-Prasad subgroups
\begin{equation*}
    G^{(j)}_{>r}\subset G^{(j)}_{\ge r}, \quad r\ge0,
\end{equation*}
we denote their arc groups by
\begin{equation*}
    (LG^{(j)})_{> r}:=L^+(G^{(j)}_{>r})\subset L^+(G^{(j)}_{\ge r})=:(LG^{(j)})_{\ge r}.
\end{equation*}

On the dual Lie algebra we define the dual lattices
\begin{equation*}
\frg^{(j),*}_{\ge r}=(\frg^{(j)}_{>-r})^{\bot}, \quad \frg^{(j),*}_{> r}=(\frg^{(j)}_{\ge -r})^{\bot}
\end{equation*}
under the residue pairing.

Let
\begin{eqnarray*}
    V^{(j)}_{>r}=V^{(j)}\cap \frg^{(j)}_{>r}, \quad V^{(j)}_{\ge r}=V^{(j)}\cap \frg^{(j)}_{\ge r},\\
    V^{(j),*}_{>r}=V^{(j),*}\cap \frg^{(j),*}_{>r}, \quad V^{(j)}_{\ge r}=V^{(j),*}\cap \frg^{(j),*}_{\ge r}.
\end{eqnarray*}
We also denote the subquotient of the Moy-Prasad filtration at step $r$ by
\begin{equation*}
    \frg^{(j)}_{=r}=\frg^{(j)}_{\ge r}/\frg^{(j)}_{>r}, \quad \frg^{(j),*}_{=r}=\frg^{(j),*}_{\ge r}/\frg^{(j),*}_{>r},
    \quad V^{(j)}_{=r}=V^{(j)}_{\ge r}/V^{(j)}_{>r}, \quad V^{(j),*}_{=r}=V^{(j),*}_{\ge r}/V^{(j),*}_{>r}.
\end{equation*}

\sss{The lattice $\frJ$}\label{sss:J}
We start by constructing a lattice $\frJ^{(j)}\subset V^{(j)}$ for each $0\le j\le d$. Set $s_j=r_j/2$ for $0\le j\le d$.
\begin{itemize}
    \item For $j=0$ let $\frJ^{(j)}=\frg^{(0)}_{\ge0}$.
    
    \item For $1\le j\le d$ and $V^{(j)}_{\ge s_{j-1}}=V^{(j)}_{> s_{j-1}}$, let $\frJ^{(j)}=V^{(j)}_{\ge s_{j-1}}$.
    
    \item When $1\le j\le d$ and $V^{(j)}_{\ge s_{j-1}}\ne V^{(j)}_{> s_{j-1}}$. We consider the alternating pairing
    \begin{eqnarray}\label{om j-1}
        \om^{(j-1)}: V^{(j)}_{=s_{j-1}}\times V^{(j)}_{=s_{j-1}} & \to& k\\
        (v,v')&\mapsto& \j{\l^{(j-1)}, [v,v']}. 
    \end{eqnarray}
    Here, residue pairing with $\l^{(j-1)}$ is a well-defined linear map  $V^{(j)}_{=r_{j-1}}\to k$ since $\l^{(j-1)}$ has pole order $r_{j-1}$. It is easily checked that  \eqref{om j-1} gives a symplectic form on $V^{(j)}_{=s_{j-1}}$.
    
    Choose a Lagrangian subspace $L^{(j)}\subset V^{(j)}_{=s_{j-1}}$, and let $\frJ^{(j)}\subset V^{(j)}_{\ge s_{j-1}}$ be the preimage of $L^{(j)}$ under the projection
    \begin{equation*}
        V^{(j)}_{\ge s_{j-1}}\surj V^{(j)}_{=s_{j-1}}.
    \end{equation*}
\end{itemize} 
Consider the following lattice in $\frg$
\begin{equation*}
\frJ:=\bigoplus_{j=0}^d\frJ^{(j)}.
\end{equation*}

\begin{lemma} With any choices of Lagrangians $L^{(j)}\subset V^{(j)}_{=s_{j-1}}$ for $1\le j\le d$, the corresponding lattice $\frJ$ is a subalgebra of $\frg$. Moreover, the residue pairing with $\l$ gives a Lie algebra homomorphism
    \begin{equation}\label{psi lam}
        \psi_\l: \frJ\to k.
    \end{equation}
    
\end{lemma}

Let $\frJ^\bot\subset L\frg^*$ be the dual lattice of $\frJ$. 

Finally, let $\bJ\subset LG$ be the connected subgroup whose Lie algebra is $\frJ$. Concretely it can be described as
\begin{equation*}
    \bJ=(LG^{(0)})_{\ge 0}(LG^{(1)})_{\ge s_0, L^{(1)}}\cdots (LG^{(d)})_{\ge s_{d-1}, L^{d-1}}.
\end{equation*}
Here $(LG^{(j)})_{\ge s_{j-1}, L^{(j)}}\subset (LG^{(j)})_{\ge s_{j-1}}$ is the preimage of $L^{(j)}$ under the projection $(LG^{(j)})_{\ge s_{j-1}}\to \frg^{(j)}_{=s_{j-1}}$. Note that $\bJ$ depends on the choice of Lagrangians $\{L^{(j)}\}_{1\le j\le d}$.

\begin{theorem}\label{th:compare JKY} Let $(M,\l)$ be a polar datum.  Then the inclusions $\fm^*_{>0}\subset \frJ^\bot$ and $(LM)_{\ge0}\subset \bJ$ induce an isomorphism of stacks
\begin{equation*}
    [(\l+\fm^*_{>0})/(LM)_{\ge0})]\isom [(\l+\frJ^\bot)/\bJ].
\end{equation*}
\end{theorem}

\begin{cor}\label{c:Kxi}
Let $(T, \l)$ be a toral polar datum, and let $\xi$ be the unique polar-cuspidal datum extending $(T,\l)$. 
Then we have an isomorphism (see \eqref{Sxi toral})
\begin{equation*}
    \cS_\xi\cong \Av_{\bJ,!}^{LG}\AS_{\psi_\l,\frJ}.
\end{equation*}
Here we are using notation from \S\ref{sss:eq FT}.
\end{cor}

\begin{remark}
    Assume the situation is defined over a subfield $F_0=k_0\lr{t}\subset F$ where $k_0\subset k$ is a finite field. Assume $T$ is an anisotropic torus over $F_0$. Then the polar datum $(T,\l)$ gives a Yu datum as explained in \S\ref{ss:Yu data}, setting $\pi_0$ to be the trivial representation of $T(F_0)$. Let $J_{\Yu}\subset G(F_0)$ be the compact open subgroup that is slightly larger than $\bJ(k_0)$ (where $\bJ\subset LG$ is constructed in \S\ref{sss:J}):
    \begin{equation*}
        J_{\Yu}=T(F_0)\bJ(k_0)
    \end{equation*}
    i.e., we replace the $(LG^{(0)})_{\ge0}=(LT)_{\ge0}$ factor of $\bJ$ by the whole $LT$.  Then Yu's supercuspidal representation is the compact induction
    \begin{equation*}
        \rho_{M,\l}:=\cInd_{\bJ_{\Yu}}^{G(F_0)}(\phi_\l)
    \end{equation*}
    where $\phi_\l$ is a character $J\to \Qlbar^\times$ extending the linear character $\frJ\xr{\psi_\l}k_0\xr{\psi\c\Tr}\Qlbar^\times$. The choices of the Lagrangians correspond to choices of Schr\"odinger models in the Weil representations that appear in the inductive steps of Yu's construction (see \cite[\S11]{Yu}).

    According to Arthur \cite{Ar},  the character of $\rho_{M,\l}$ can be calculated using the weighted orbital integral of the function $\phi_\l$ supported on $J_{\Yu}$. On the other hand, under the sheaf-to-function correspondence, $\pi_{\bJ!}\AS_{\psi_\l}$ corresponds exactly to the Lie algebra analog of such weighted orbital integrals (modulo the slight difference between $\bJ(k_0)$ and $J_{\Yu}$). In this sense, the character sheaf $\cK_\xi$ is a geometrization of a Lie algebra analog of the character of the supercuspidal representation $\rho_{(M,\l)}$.
\end{remark}

\begin{remark}\label{r:Kr} Theorem \ref{th:compare JKY} has a variant that does not involve extra choice of Lagrangians that was needed in the construction of $\frJ$. For a polar datum $(M,\l)$, let
\begin{equation*}
    \bK(\orr{r})=(LM)_{\ge0}(LG^{(1)})_{>r_0}
    \cdots (LG^{(d-1)})_{> r_{d-2}} (LG^{(d)})_{> r_{d-1}}.
\end{equation*}
Then the inclusions induce an isomorphism of stacks
\begin{equation*}
    [(\l+\fm^{*}_{>0})/(LM)_{\ge0}]\isom [(\l+\frg^{*}_{>0})/\bK(\orr{r})].
\end{equation*}
\end{remark}

\subsection{Epipelagic character sheaves} We illustrate how Theorem \ref{th:compare JKY} works out for the epipelagic polar datum $(T,\l)$ as defined in Example \ref{ex:epi}.

For simplicity, we assume
\begin{equation*}
    \mbox{$G$ is split and almost simple, and $T$ is anisotropic.}
\end{equation*}
Fix an isomorphism $G=G_0\times_k F$ for a reductive group $G_0$ over $k$. Recall from Example \ref{ex:epi} that $T$ corresponds to a regular injective homomorphism $\ov w:\mu_m\to \WW$. The assumption that $T$ is anisotropic implies $m$ is the order of an {\em elliptic} regular element of $\WW$. Indeed, $m$ determines the $\WW$-conjugacy class of $\ov w$, hence the conjugacy class of $T$. 

By construction, $\l=vd\t/\t^2$ (where $v\in \frt^*_0(1)$ is regular) lies in the degree $-1/m$ piece of the Moy-Prasad grading of $\frt^*$. The twisted Levi sequence attached to $(T,\l)$ has only two steps
\begin{equation*}
    G^{(0)}=T\subset G^{(1)}=G
\end{equation*}
with $r_0=1/m$. 

Now we give an explicit construction of the epipelagic polar datum in the conjugacy class of $(T,\l)$. 

Let $T_0\subset G_0$ be a  maximal $k$-torus, giving an apartment $\frA\cong \xcoch(T_0)_\RR\subset \frB(G)$. Let $x=\rho^\v/m\in \frA$ and use it to define the Moy-Prasad filtration $\frg_{x,\ge r}$ on $L\frg$. In fact we have a $\ZZ$-grading
\begin{equation}\label{MP gr}
L\frg=\wh\bigoplus_{i\in \ZZ}(L\frg)(i)
\end{equation}
where $(L\frg)(i)$ is spanned by the affine root spaces corresponding to affine roots $\a$ (viewed as affine linear functions on $\frA$, including imaginary roots) such that $\a(\r^{\vee}/m)=i/m$ (when $i=0$ we also include $\frt_0=\Lie T_0$). Under \eqref{MP gr}, $\frg_{x,\ge r }$ corresponds to $\prod_{i/m\ge r}(L\frg)(i)$, and $\frg_{x,> r }$ corresponds to $\prod_{i/m> r}(L\frg)(i)$.

Each $(L\frg)(i)$ is in fact contained in the polynomial loop Lie algebra $\frg_0[t,t^{-1}]$. Evaluating $t$ at $1$ gives an embedding $\ev_{t=1}: (L\frg)(i)\subset \frg_0[t,t^{-1}]\to \frg_0$ whose image depends only on $i$ mod $m$. We denote this image by $\frg(\un i)$ (where $\un i\in \ZZ/m\ZZ$ is the residue class of $i$ mod $m$). We get a $\ZZ/m\ZZ$-grading on $\frg_0$
\begin{equation*}
\frg_0=\bigoplus_{\un i\in \ZZ/m\ZZ}\frg_0(\un i).
\end{equation*}
Dually we have a $\ZZ$-grading on $L\frg^*$
\begin{equation*}
L\frg=\wh\bigoplus_{i\in \ZZ}(L\frg^*)(i)
\end{equation*}
such that $(L\frg^*)(i)$ is in perfect pairing with $(L\frg)(-i)$ under the residue pairing. It induces a
$\ZZ/m\ZZ$-grading on $\frg^*_0$
\begin{equation*}
    \frg^*_0=\bigoplus
    _{\un i\in \ZZ/m\ZZ}\frg^*_0(\un i).
\end{equation*}

By \cite[Theorem 3.2.5]{OY}, $\frg^*_0(-\un 1)$ contains a regular semisimple element if and only if $m$ is the order of a regular element in $\WW$. Take a regular semisimple element $\ov\l'\in \frg^*_0(-\un 1)$, and let $\l'\in (L\frg^*)(-1)$ be the corresponding element under the canonical isomorphism $(L\frg^*)(-1)\isom \frg^*_0(-\un 1)$. Let $T'=C_G(\l')$. Then 
one can show that $T'$  is $G(F)$-conjugate to $T$ (because they both correspond to a regular homomorphism $\mu_m\incl \WW$ which is unique up to conjugacy). Moreover, one can choose $\l'\in (L\frg^*)(-1)$ to have the same image in $L\frc^*$ as $\l$, so that the toral polar data $(T,\l)$ and $(T',\l')$ are $G(F)$-conjugate. Therefore we may assume $(T,\l)=(T',\l')$.


Let $\bP=(LG)_{x,\ge0}$ be the 
parahoric subgroup corresponding to $x$, and $\bP^+=(LG)_{x,>0}=(LG)_{x,\ge 1/m}$ its pro-unipotent radical.  The construction of $\bJ$ in \S\ref{sss:J} in this case gives
\begin{equation*}
    \bJ=\bP^+.
\end{equation*}
We have
\begin{equation*}
    \frJ=\Lie \bJ=\frg^*_{x,>0}, \quad \frJ^\bot=\frg^*_{x,\ge 0}.
\end{equation*}
The element $\l\in L\frg^*(-1)$ defines a linear character
\begin{equation*}
    \psi_\l: \frp^+=\frg_{x,\ge 1/m}\to \frg_{x, =1/m}=(L\frg)(1)\xr{\j{\l,-}} k.
\end{equation*}
In this case,  Theorem \ref{th:compare JKY}  asserts that inclusions induce an isomorphism
\begin{equation*}
    [(\l+\frt^*_{tn})/\bT^+]\isom [(\l+\frg^*_{x,\ge 0})/\bP^+].
\end{equation*}
Note that $\bT^+=\bT$ and $\frt^*_{tn}=\frt^*_{x,>0}=\frt^*_{x,\ge 0}$ because $T$ assumed to be  anisotropic.

Let $\xi$ be the unique polar-cuspidal datum extending $(T,\l)$. 
Corollary \ref{c:Kxi} in this case says that
\begin{equation}\label{epi CS}
    \cS_\xi\cong \Av_{\bP^{+}!}^{LG}\AS_{\psi_\l, \frp^+}.
\end{equation}
We call the above sheaf
the {\em epipelagic character sheaf} attached to $(\bP,\l)$. 

Since the image of $\pi_{\bP^{+}}$ only meets topologically nilpotent elements, $\cS_\xi$ is supported on $(L\frg)_{tn}$.

\begin{exam}
When $G$ is almost simple, let $m=h$ be its Coxeter number. Then $\bP=\bI$ is an Iwahori subgroup of $LG$, and $L\frg^*(-1)$ is the direct sum of negative affine simple root spaces $(L\frg^*)(-\a_{i})$ for affine simple roots $\{\a_{0},\a_{1},\cdots, \a_{r}\}$ of $G$. Take $\l\in L\frg^*(-1)$ to have nonzero coordinate in each line $(L\frg^*)(-\a_{i})$. Then $\l$ is regular semisimple, and $(T=C_G(\l), \l)$ is an epipelagic polar datum. The corresponding epipelagic character sheaf $\cS_\xi$ is the $!$-direct image of $\AS_{\psi_\l}$ under $\pi_{\bI^+}:[\frI^+/\bI^+]\to [L\frg/LG]$.
\end{exam}

\subsection{Expected properties of character sheaves on $L\frg$}

\sss{Constructibility} As we have seen in Example \ref{ex:toral coCS}, the stalks of the complex $\cK_{\xi, \bQ}$ is not finite-dimensional. Therefore we do not expect their Fourier transform to be constructible. 

However, when $T$ is an elliptic maximal torus, and $\xi$ is the unique polar-cuspidal datum extending a polar datum $(T,\l)$, we expect that the sheaf $\cK_{\xi}$ in \eqref{Kxi toral}, as well as its Fourier transform $\cS_\xi$, to be constructible.

\sss{Perversity}
Let $L\frg^{\hs}\subset L\frg$ be the regular semisimple locus: it is the complement of $L\frD$ where $\frD\subset \frg$ is discriminant divisor.

In \cite[7.2]{BKV}, Bouthier--Kazhdan--Varshavsky define a perverse $t$-structure on the regular semisimple locus $L\frg^{\hs}/LG\subset L\frg/LG$ by assigning a perversity to each {\em root valuation stratum} in the sense of Goresky-Kottwitz-MacPherson \cite{GKM}. They show that the affine Springer sheaf $\cS_{\xi_0,\bI}$ (see Example \ref{ex:aff Spr}) is perverse in their $t$-structure. 

We make the following conjecture. 

\begin{conj} For any polar-cuspidal datum $\xi=(M,\l,L,\cO,\cL)$ for $G$, and any parahoric subgroup $\bQ\subset LM$ containing $L$ as a Levi factor, the object
\begin{equation*}
\cS_{\xi,\bQ}|_{L\frg^{\hs}}\in D^\bu_{LG}(L\frg^{\hs})
\end{equation*}
is perverse in the sense of Bouthier--Kazhdan--Varshavsky.
\end{conj}
We have checked that the conjecture holds when $(M,\l)=(G,0)$. The epipelagic character sheaves $\cS_\xi$ in \eqref{epi CS} are the next open cases, for which we only verified the case of $\SL_2$.

\sss{Inner structure of the blocks}
By \eqref{CS reduce to M}, the study of a general block in $\CS(L\frg)$ can be reduced to the study of a nilpotent block of a twisted Levi subgroup. We expect that each nilpotent block (hence any block) of $\CS_{\xi}(L\frg)$ is equivalent to a category of modules over a certain double affine Hecke algebra. All this is parallel to the generalized Springer correspondence, and to the known properties of the  usual Bernstein blocks for $p$-adic groups.

\end{document}